\documentclass[a4paper, reqno]{amsart}

\calclayout

\usepackage{microtype}
\usepackage{amscd}
\usepackage{amsmath,amssymb,amsfonts}
\usepackage{xstring}
\usepackage{xcolor}
\usepackage[mathscr]{euscript}
\usepackage{xypic}
\usepackage[cmtip,all,2cell]{xy}
\entrymodifiers={+!!<0pt,\fontdimen22\textfont2>}
\usepackage{url}
\usepackage{latexsym}
\usepackage{amsthm}
\usepackage[latin1]{inputenc}
\usepackage{enumitem}
\usepackage[colorlinks=true, linkcolor={blue!50!black}, pdfhighlight=/O, ocgcolorlinks=true]{hyperref}
\usepackage{graphicx}
\usepackage{color}

\numberwithin{equation}{section}
\theoremstyle{plain}
\newtheorem{theorem}[equation]{Theorem}
\newtheorem*{ntheorem}{Theorem}
\newtheorem{lemma}[equation]{Lemma}
\newtheorem{proposition}[equation]{Proposition}
\newtheorem{corollary}[equation]{Corollary}

\theoremstyle{definition}
\newtheorem{definition}[equation]{Definition}
\newtheorem{variant}[equation]{Variant}
\newtheorem{construction}[equation]{Construction}
\newtheorem{example}[equation]{Example}
\newtheorem{remark}[equation]{Remark}
\newtheorem*{nremark}{Remark}
\newtheorem*{warning}{Warning}

\makeatletter
\renewcommand\paragraph{\@startsection{paragraph}{4}{\z@}%
                                    {2.25ex \@plus1ex \@minus.2ex}%
                                    {-0.5em}%
                                    {\normalfont\normalsize\bfseries}}
\makeatother

\newcommand{\mf}[1]{\mathfrak{#1}}
\newcommand{\mc}[1]{\mathcal{#1}}
\newcommand{\ul}[1]{\underline{#1}}
\newcommand{\ol}[1]{\overline{#1}}
\newcommand{\cat}[1]{
\StrLen{#1}[\mystrlen]
\ifnum\mystrlen=1 \mathscr{#1}
\else \mathrm{#1}
\fi}
\newcommand{\ope}[1]{
\StrLen{#1}[\mystrlengt]
\ifnum\mystrlengt=1 \mathscr{#1}
\else \mathrm{#1}
\fi
}
\newcommand{\Set}[0]{\cat{Set}}

\newcommand{\colim}{\operatornamewithlimits{\mathrm{colim}}}
\newcommand{\hocolim}{\operatornamewithlimits{\mathrm{hocolim}}}

\newcommand{\dau}[0]{\partial}

\newcommand{\mm}[1]{\mathrm{#1}}
\newcommand{\Hom}[0]{\mm{Hom}}
\newcommand{\Map}[0]{\mm{Map}}
\renewcommand{\rto}[1]{\stackrel{#1}{\rt}}

\newcommand{\rt}[0]{\rightarrow}
 
\newcommand{\spec}[0]{\mm{Spec}}
\newcommand{\coo}[0]{\mc{C}^\infty}

\newcommand{\Fun}[0]{\cat{Fun}}
\newcommand{\Del}[0]{\mm{\Delta}}
\newcommand{\op}[0]{\mm{op}}
\newcommand{\dR}[0]{\mm{dR}}
\newcommand{\Mod}[0]{\mm{Mod}}
\newcommand{\dgMod}[0]{\Mod^{\dg}}
\newcommand{\sbullet}[0]{{\scriptscriptstyle \bullet}}
\newcommand{\Rep}[0]{\cat{Rep}}
\newcommand{\dgRep}[0]{\Rep^{\dg}}
\newcommand{\sgeq}[0]{{\scriptscriptstyle \geq}}

\newcommand{\dgCAlg}[0]{\CAlg^\mm{\dg}}
\newcommand{\dg}[0]{\mm{dg}}
\newcommand{\CAlg}[0]{\cat{CAlg}}
\newcommand{\LieAlgd}[0]{\cat{LieAlgd}}
\newcommand{\dgLieAlgd}[0]{\LieAlgd^{\dg}}
\newcommand{\LooAlgd}[0]{\cat{L_\infty Algd}}
\newcommand{\dgLooAlgd}[0]{\LooAlgd^{\dg}}
\newcommand{\dgLooAlg}[0]{\cat{L_\infty Alg}^{\dg}}
\newcommand{\dgLie}[0]{\cat{Lie}^{\dg}}
\newcommand{\Env}[0]{\ope{Env}}
\newcommand{\rEnv}[0]{\ol{\Env}}

\title{Homotopical algebra for Lie algebroids}
\author{Joost Nuiten}
\email{joost.nuiten@gmail.com}
\address{Mathematical Institute\\ Utrecht University\\ P.O. Box 80010\\ 3508 TA Utrecht\\ The Netherlands.}
\date{February 1, 2019}
\thanks{This is a post-peer-review, pre-copy-edit version of an article published in Applied Categorical Structures. The final authenticated version is available on-line at: \url{https://doi.org/10.1007/s10485-019-09563-z}}

\begin{document}
\begin{abstract}
We construct Quillen equivalent semi-model structures on the categories of dg-Lie algebroids and $L_\infty$-algebroids over a commutative dg-algebra in characteristic zero. This allows one to apply the usual methods of homotopical algebra to dg-Lie algebroids: for example, every Lie algebroid can be resolved by dg-Lie algebroids that arise from dg-Lie algebras, i.e.\ whose anchor map is zero. As an application, we show how Lie algebroid cohomology is represented by an object in the homotopy category of dg-Lie algebroids.
\end{abstract}

\maketitle

\section{Introduction}
The purpose of this paper is to provide a model-categorical description of the homotopy theory of \emph{differential graded Lie algebroids} over a commutative dg-algebra of characteristic zero. Just as Lie algebroids are frequently used to describe infinitesimal structures in algebraic \cite{rin63} or differential geometry \cite{mac87}, such dg-Lie algebroids can be used to describe infinitesimal structures in \emph{derived} geometry. In particular, our main reason to develop a homotopy theory of dg-Lie algebroids is to use it to study the role of dg-Lie algebroids in deformation theory. 

To motivate this, let us start by considering dg-Lie algebroids over a field $k$ of characteristic zero, which are simply dg-Lie algebras. A fundamental principle in deformation theory, tracing back to the work of Deligne and Drinfeld, asserts that for any point in a moduli space over $k$, its formal neighbourhood is controlled by a dg-Lie algebra. This idea has proven to give a very effective and concrete method for describing the infinitesimal behaviour of moduli spaces, which applies in many situations (see e.g.\ \cite{gol88,hin98,kon03}). Nowadays, a precise formulation of the equivalence between dg-Lie algebras and formal deformation problems is provided in terms of \emph{homotopy theory}: work of Pridham \cite{pri10} and Lurie \cite{lur11X} establishes an equivalence between the homotopy theory of dg-Lie algebras and a certain homotopy theory of formal moduli problems over $k$.

One can try to extend these ideas to more general commutative dg-algebras $A$ of characteristic zero: given a map $\spec(A)\rt X$ from an affine (derived) scheme to a moduli space, one can try to describe a formal neighbourhood of $\spec(A)$ inside $X$ in terms of a dg-Lie algebroid over $A$. The recent work of Gaitsgory and Rozenblyum \cite{gai16} makes extensive use of this viewpoint on Lie algebroids: they essentially \emph{define} Lie algebroids to be formal moduli problems over $\spec(A)$ and develop their theory in these terms. 

This paper serves as a complement to this work, and provides a rigid, point-set model for the homotopy theory of Lie algebroids in terms of a dg-version of the usual notion of a Lie algebroid \cite{rin63}. More precisely, the main result of this paper is the following:
\begin{ntheorem}
Let $A$ be a commutative dg-algebra over a field of characteristic zero. The category of dg-Lie algebroids carries a semi-model structure, in which a map is a weak equivalence (fibration) if it is a quasi-isomorphism (degreewise surjective). 
\end{ntheorem}
\begin{warning}
It has been asserted in \cite{vez13} that the category of dg-Lie algebroids carries an actual model structure. However, there is a gap in the argument of loc.\ cit: it does not check whether any dg-Lie algebroid can be replaced by a fibrant dg-Lie algebroid. This condition often holds trivially, e.g.\ for algebras over dg-operads, but it does \emph{not} hold for the category of dg-Lie algebroids (see Example \ref{ex:counter}). In particular, this counterexample shows that, contrary to \cite{vez13}, the quasi-isomorphisms and surjections \emph{do not define an actual model structure} on the category of dg-Lie algebroids. 
\end{warning}
Our proof method is based on an analysis of pushouts of generating trivial cofibrations of dg-Lie algebroids. This analysis proceeds along the same lines as for algebras over operads, with one exception: the pushout of a generating trivial cofibration need not be an injection in general. This makes our argument slightly more complicated than usual, and is also the main reason why dg-Lie algebroids do not form a genuine model category (see Example \ref{ex:counter}). 
\begin{nremark}[Global approach]
Instead of considering dg-Lie algebroids over a fixed $A$, one can also consider the category of pairs $(A, \mf{g})$ where $A$ is a cdga and $\mf{g}$ is a dg-Lie algebroid over $A$. Such pairs are the algebras over a certain 2-coloured operad \cite{laa04} and therefore admit a model structure. This `global' model structure is not so useful for the deformation-theoretic applications in \cite{nui17c}. For instance, the functor sending $(A, \mf{g})\mapsto A$ is not particularly well-behaved (i.e.\ it is not a (op-)fibration), so the global model structure does not restrict to its fibres.
\end{nremark}
In \cite{nui17c}, we show that the homotopy theory of dg-Lie algebroids over (cofibrant) $A$ provided by the above theorem is equivalent to the homotopy theory of formal moduli problems over $A$. Informally, this means that dg-Lie algebroids can indeed be used as algebraic models for the formal neighbourhoods of $\spec(A)$ inside moduli spaces. The fact that dg-Lie algebroids indeed behave like `algebraic objects' from the viewpoint of homotopy theory is made precise by the following result:
\begin{ntheorem}
Consider the forgetful functor $\dgLieAlgd_A\rt \dgMod_A/T_A$, which sends a dg-Lie algebroid to its anchor map. This is a right Quillen functor that detects weak equivalences and preserves all sifted homotopy colimits.
\end{ntheorem}
Concretely, this provides a simple description of the homotopy colimit of a simplicial diagram of dg-Lie algebroids (see Example \ref{ex:totalisation}). On a more formal level, this theorem guarantees that the homotopy theory of dg-Lie algebroids is well behaved. For example, the free-forgetful adjunction $\dgMod_A/T_A\leftrightarrows \dgLieAlgd_A$ induces a monadic adjunction between $\infty$-categorical localisations \cite[Section 4.7]{lur16}. This formal property of the homotopy theory of dg-Lie algebroids plays an important technical role in \cite{nui17c} (cf.\ also \cite[Section 2]{lur11X}). Similarly, the right Quillen functor to $A$-linear dg-Lie algebras
$$\xymatrix{
\dgLieAlgd_A\ar[r] & \dgLie_A;\hspace{4pt} \mf{g}\ar@{|->}[r] & \ker\Big(\rho\colon \mf{g}\rt T_A\Big)
}$$
also detects weak equivalences and preserves sifted homotopy colimits. This means that dg-Lie algebroids can be considered as algebras for some monad on the $\infty$-category of $A$-linear Lie algebras (cf.\ \cite{gai16}). In particular, it follows that every dg-Lie algebroid admits a simplicial resolution by Lie algebroids arising from Lie algebras, i.e.\ whose anchor map is null-homotopic (by Lie algebroid maps). Such resolutions can be used to reduce statements about dg-Lie algebroids to the case of dg-Lie algebras, which is sometimes easier.

\paragraph{Outline}
The paper is outlined as follows: after recalling some preliminaries in Section \ref{sec:prelim}, we describe the above theorems, together with some variants and immediate consequences, in Section \ref{sec:main}. For example, both theorems have analogues for $L_\infty$-algebroids as well. The proofs of the theorems are contained in Section \ref{la:sec:technicalities}.

In the remainder of the text, we discuss a few model-categorical tools that one can use to study dg-Lie algebroids. For example, we give a concrete cofibrant replacement of dg-Lie algebroids and $L_\infty$-algebroids in Section \ref{la:sec:cofibrantreplacement}, which can be used (in certain situations) to present derived mapping spaces in terms of `$\infty$-morphisms'. In Section \ref{sec:appl}, we use this to identify the (reduced) Lie algebroid cohomology (see e.g.\ \cite{rin63})
$$
\ol{\mm{H}}{}^n(\mf{g}; A) \cong \big[\mf{g}, T_A\oplus A[n-1]\big]
$$
with the set of homotopy classes of maps into the square zero extension of the tangent Lie algebroid by a shifted copy of $A$. Similarly, we show that the Weil algebra of a Lie algebroid $\mf{g}$ (see e.g.\ \cite{aba11}) simply computes the Lie algebroid cohomology of its free loop space $\mc{L}\mf{g}$. The mixed structure on the Weil algebra can then be thought of as arising from the $S^1$-action on $\mc{L}\mf{g}$ by rotation of loops.

\paragraph{Conventions} Throughout, we work over a field $k$ of characteristic zero, so that a chain complex always means a chain complex of $k$-vector spaces. We use homological conventions for chain complexes, i.e.\ the differential $\dau$ is of degree $-1$. When $V$ is a chain complex, we denote its $n$-fold suspension by $V[n]$ and its cone by $V[0, 1]$.

\paragraph{Acknowledgements}
I would like to thank the anonymous referee, whose many questions and comments have helped greatly improving the clarity and exposition of the paper. This work was supported by the Nederlandse Organisatie voor Wetenschappelijk Onderzoek (NWO).

\section{Preliminaries}\label{sec:prelim}
The purpose of this section is twofold: on the one hand, we will recall the notion of a dg-Lie algebroid over a commutative dg-algebra and on the other hand, we will recall the notion of a (left) semi-model category.

\paragraph{DG-Lie-algebroids}
Throughout, fix a (unital) commutative dg-algebra $A$ over $k$ and let $T_A = \mm{Der}_k(A, A)$ be the chain complex of derivations of $A$. This complex carries the structure of a dg-$A$-module, given by pointwise multiplication in $A$, as well as the structure of a dg-Lie-algebra over $k$, with Lie bracket given by the commutator bracket. Following \cite{vez13}, we define a dg-Lie algebroid over $A$ as follows:
\begin{definition}\label{la:def:lralgebra}
A \emph{dg-Lie algebroid} $\mf{g}$ over $A$ is an (unbounded) dg-$A$-module $\mf{g}$, equipped with a $k$-linear dg-Lie algebra structure and an \emph{anchor map} $\rho\colon \mf{g}\rt T_A$ such that
\begin{enumerate}
\item $\rho$ is both a map of dg-$A$-modules and dg-Lie algebras.
\item the failure of the Lie bracket to be $A$-bilinear is governed by the Leibniz rule
$$
[x, a\cdot y] = (-1)^{|x|\cdot |a|} a[x, y] + \rho(x)(a)\cdot y.
$$
\end{enumerate}
Let $\dgLieAlgd_A$ be the category of dg-Lie algebroids over $A$, with maps given by $A$-linear maps over $T_A$ that preserve the Lie bracket.
\end{definition}
These objects are also known as differential graded \emph{Lie-Rinehart algebras}, which were originally considered in \cite{rin63} in the non-dg setting.
\begin{example}\label{ex:atiyah}
Any dg-$A$-module $E$ gives rise to an \emph{Atiyah dg-Lie algebroid} $\mm{At}(E)$ over $A$, which can be described as follows: a degree $n$ element of $\mm{At}(E)$ is a tuple $(v, \nabla_v)$ consisting of a derivation $v\colon A\rt A$ (of degree $n$), together with a $k$-linear map $\nabla_v\colon E\rt E$ (of degree $n$) such that
$$
\nabla_v(a\cdot e) = v(a)\cdot e + (-1)^{|a|\cdot n} a\cdot \nabla_v(e)
$$
for all $a\in A$ and $e\in E$. This becomes a dg-$A$-module under pointwise multiplication and a dg-Lie algebra under the commutator bracket. The anchor map is the obvious projection $\mm{At}(E)\rt T_A$ sending $(v, \nabla_v)$ to $v$.
\end{example}
\begin{example}
Similarly, suppose that $E\in \dgMod_A$ has the structure of an algebra over a $k$-linear dg-operad $\ope{P}$. Then there is a sub dg-Lie algebroid of $\mm{At}(E)$ consisting of the tuples $(v, \nabla_v)$ where $\nabla_v$ is a $\ope{P}$-algebra derivation.
\end{example}

\paragraph{$L_\infty$-algebroids}
One frequently encounters situations in which a chain complex does not admit a dg-Lie algebra structure, but instead has a homotopy coherent Lie algebra structure, i.e.\ an $L_\infty$-structure.
\begin{definition}[{cf.\ \cite[Section 13.2.12]{lod12}}]
Let $\mf{g}$ be a chain complex. An \emph{$L_\infty$-structure} on $\mf{g}$ is given by a collection of graded anti-symmetric maps $[-, ..., -]\colon \mf{g}^{\otimes n}\rt \mf{g}$ of (homological) degree $n-2$, for each $n\geq 2$, satisfying a sequence of Jacobi identities
$$
J^k(x_1, \cdots, x_k)=0 \qquad \quad k\geq 2.
$$
Here $J^k(x_1, \cdots, x_k)$ is the $k$-th \emph{Jacobiator}, given by
\begin{equation}\label{eq:jacobiator} \mathop{\sum}\limits_{i+j = k+1} (-1)^{i(j-1)} \mathop{\sum}\limits_{\sigma} (-1)^{\sigma}\pm \big[\big[x_{\sigma(1)}, \cdots , x_{\sigma(i)}\big]_i, x_{\sigma(i+1)}, \cdots, x_{\sigma(k)}\big]_{j},
\end{equation}
where $\sigma$ runs over the $(i, k-i)$-unshuffles, $\pm$ denotes the usual Koszul sign due to the permutation of the variables $x_i$ and the 1-ary bracket $[-]\colon \mf{g}\rt\mf{g}$ is given by $[x]=\dau_{\mf{g}[1]}x = -\dau_\mf{g} x$.
\end{definition}
\begin{remark}\label{rem:nl}
By a (linear) \emph{map of $L_\infty$-algebras} $\mf{g}\rt \mf{h}$ we will mean a map of chain complexes preserving the brackets, i.e.\ a map of algebras over the $L_\infty$-operad. There is a more general notion of map between $L_\infty$-algebras, known as an $\infty$-morphism or $L_\infty$-morphism. We will come back to this in Section \ref{la:sec:cofibrantreplacement}, cf.\ Definition \ref{la:def:nonlinearmaplooalgebras}.
\end{remark}
There is a map of operads $L_\infty\rt \mm{Lie}$, realising a Lie algebra as an $L_\infty$-algebra whose $n$-ary brackets vanish for $n\geq 3$. This map is a quasi-isomorphism, so that any $L_\infty$-algebra is quasi-isomorphic to a dg-Lie algebra, which is however typically \emph{much larger}. Conversely, any chain complex that is quasi-isomorphic to a dg-Lie algebra can be endowed with an $L_\infty$-structure, by the homotopy transfer theorem \cite[Section 10.2]{lod12}.

Similarly, there are situations where a given dg-$A$-module $\mf{g}$ may not support a strict dg-Lie algebroid structure, but does admit some homotopy coherent refinement thereof. Such homotopy coherent structures can be useful, for example, when $\mf{g}$ arises from a complex of \emph{vector bundles}, i.e.\ a complex of finite rank projective $A$-modules (cf.\ Remark \ref{rem:finiterank}). A \emph{strict} dg-Lie algebroid model for $\mf{g}$ can then often be found only after replacing $\mf{g}$ by a much larger quasi-isomorphic complex, which is more complicated to understand geometrically (cf.\ Corollary \ref{la:cor:equivalencetohomotopyLRalgebras}).

There are several possible homotopy coherent weakenings of the structure of a dg-Lie algebroid, depending on which part of the structure one wants to relax. For example, one can weaken the Lie bracket, the anchor map (see Remark \ref{rem:liealgdwithnlanchor}) or the $A$-module structure. For our model-categorical purposes, it is convenient to only weaken the Lie algebra structure to an $L_\infty$-structure.

\begin{definition}[{see e.g.\ \cite[Definition 4]{bon13}, \cite[Definition 2.1]{pym16}}]\label{la:def:loorinehartalgebras}
We define an \emph{$L_\infty$-algebroid} over $A$ to be a dg-$A$-module $\mf{g}$, equipped with the structure of a ($k$-linear) $L_\infty$-algebra and an \emph{anchor map} $\rho\colon \mf{g}\rt T_A$, such that
\begin{enumerate}
\item $\rho$ is a map of dg-$A$-modules and a (linear) map of $L_\infty$-algebras.
\item the brackets satisfy the Leibniz rules
\begin{align}\label{la:eq:leibniz}
[x, a\cdot y] &= (-1)^{ax} a[x, y] + \rho(x)(a)\cdot y\\
[x_1, ..., a\cdot x_n] &= (-1)^{an}(-1)^{a(x_1+...+x_{n-1})} a[x_1, ..., x_n] & \quad\quad n\geq 3.\nonumber
\end{align}
\end{enumerate}
Let $\dgLooAlgd_A$ be the category of $L_\infty$-algebroids over $A$, with maps given by $A$-linear maps over $T_A$ that are also (linear) maps of $L_\infty$-algebras.
\end{definition}
To obtain a well-behaved category of $L_\infty$-algebroids, e.g.\ with limits and colimits, it is necessary to only work with linear, structure-preserving maps between $L_\infty$-algebroids. There is also a more homotopy coherent notion of \emph{$\infty$-morphism} between $L_\infty$-algebroids. We will come back to this in Section \ref{la:sec:cofibrantreplacement}.
\begin{example}
$L_\infty$-algebroids arise naturally as extensions of Lie algebroids classified by higher Lie algebroid cocycles \cite{she17} (see Example \ref{la:ex:loofromcocycle} for more details).
\end{example}
\begin{example}[Action $L_\infty$-algebroids]\label{la:ex:actionalgebroids}
Let $\rho\colon \mf{g}\rt T_A$ be a map of $L_\infty$-algebras over $k$. Then $A\otimes \mf{g}$ has the structure of an $L_\infty$-algebroid, with anchor map given by the $A$-linear extension of $\rho$ and with brackets given by
\begin{align*}
[a\otimes x, b\otimes y] &= \pm ab\otimes [x, y] + a\cdot \rho(x)(b)\otimes y - (\pm) b\cdot \rho(y)(a)\otimes x \\
[a_1 \otimes x_1, ..., a_n \otimes x_n] & = \pm a_1...a_n\otimes [x_1, ..., x_n].
\end{align*}
where $\pm$ is the usual Koszul sign. The only non-trivial condition to verify is the Jacobi identity; for this it suffices to show that each Jacobiator $J^k$ is $A$-linear in each of its variables, which can be done by explicit computation.
\end{example}
\begin{example}\label{ex:ideal}
Suppose that $\mf{g}$ is an $L_\infty$-algebroid and let $J\subseteq \mf{g}$ be a \emph{dg-ideal}, i.e.\ a sub-$A$-module of the kernel of the anchor map which is closed under the brackets and the differential. Then $\mf{g}/J$ inherits the structure of an $L_\infty$-algebroid.
\end{example}
\begin{remark}\label{rem:finiterank}
Definition \ref{la:def:loorinehartalgebras} is often used in the following setting. Let $A$ be an ordinary $k$-algebra (concentrated in degree zero) and $\mf{g}$ a complex of finite rank projective $A$-modules, concentrated in degrees $\geq 0$. An $L_\infty$-algebroid structure on $\mf{g}$ is then equivalent to the data of a differential on the graded symmetric algebra $\mm{Sym}_A(\mf{g}[1]^\vee)$ \cite[Theorem 2]{bon13}; this differential is just the Chevalley-Eilenberg differential of Equation \eqref{fm:eq:cediff}. 

In the differential-geometric case where $A=\coo(M)$ is the ring of smooth functions on a manifold, this means that $\mf{g}$ is the dg-module of sections of a complex of vector bundles over $M$. The associated graded-free commutative dg-algebras are then also known as `dg-manifolds' or `NQ-supermanifolds' \cite{ale97}, which appear frequently in differential geometry and mathematical physics (see e.g.\ \cite{roy02,sev05}).
\end{remark}
\begin{remark}\label{rem:liealgdwithnlanchor}
A further generalisation of Definition \ref{la:def:loorinehartalgebras} appearing in the literature is the notion of a \emph{sh-Lie-Rinehart algebra} \cite{hue17,kje01,vit14}. Here the anchor map is allowed to be an $\infty$-morphism of $L_\infty$-algebras (as in Definition \ref{la:def:nonlinearmaplooalgebras}). Concretely, this means that $\mf{g}$ acts on $A$ in a homotopy coherent way, encoded by a collection of maps of degree $n-2$
$$\xymatrix{
\rho^{(n)}\colon \mf{g}^{n-1}\otimes A\ar[r] & A & &  n\geq 2
}$$
subject to a sequence of equations. When $n\geq 3$, the higher action map $\rho^{(n)}$ contributes an extra term to the right hand side of the Leibniz rule. 

Definition \ref{la:def:loorinehartalgebras} is the special case where we force $\rho^{(n)}=0$ for $n\geq 3$, so that the Leibniz rule reduces to the simple form of \eqref{la:eq:leibniz} when $n\geq 3$. In the situation of Remark \ref{rem:finiterank}, the $\rho^{(n)}$ vanish for degree reasons and the two notions coincide.

We have chosen to work with Definition \ref{la:def:loorinehartalgebras} instead of sh-Lie-Rinehart algebras because the latter form a category which is ill-behaved for model-categorical purposes; it lacks a terminal object, for example.
\end{remark}
The categories of dg-Lie algebroids and $L_\infty$-algebroids over $A$ fit into a commuting diagram
\begin{equation}\label{la:diag:forgetfulfunctors}\vcenter{\xymatrix{
\dgLieAlgd_A\ar[d]\ar[r] & 
\dgLooAlgd_A\ar[r] \ar[d] & \dgMod_A/T_A\ar[d]\\
\dgLie_k/T_A\ar[r] & \cat{L_\infty Alg}^{\dg}_k/T_A\ar[r] & \dgMod_k/T_A.
}}\end{equation}
The vertical functors forget the $A$-module structure, the left two horizontal functors are inclusions and the right two horizontal functors forget the $L_\infty$-structure. Each of these forgetful functors admits a left adjoint for formal reasons. In fact, the left adjoints to the vertical functors are easily identified: 
\begin{lemma}\label{lem:actionisfree}
The left adjoints to the forgetful functors 
$$
\dgLieAlgd_A\rt \dgLie_k/T_A \qquad \text{and}\qquad \smash{\dgLooAlgd_A\rt \cat{L_\infty Alg}^\dg_k/T_A}
$$
are given by the `action $L_\infty$-algebroid' construction of Example \ref{la:ex:actionalgebroids}. 
\end{lemma}
\begin{proof}
Let $\mf{g}\rt T_A$ be a $k$-linear $L_\infty$-algebra over $T_A$ and let $\mf{h}$ be an $L_\infty$-algebroid. A $k$-linear map $f\colon \mf{g}\rt \mf{h}$ over $T_A$ determines a unique $A$-linear map $g\colon A\otimes \mf{g}\rt \mf{h}$, which preserves the brackets if and only if $f$ preserves the brackets.
\end{proof}
\begin{definition}
We will denote by
$$\xymatrix{
\mm{Free}\colon \dgMod_k/T_A\ar[r] & \dgLieAlgd_A & F\colon \dgMod_A/T_A\ar[r] & \dgLieAlgd_A
}$$
the functors taking the free dg-Lie algebroid on a $k$-linear (resp.\ $A$-linear) map $V\rt T_A$. The same notation is employed for free $L_\infty$-algebroids.
\end{definition}
\begin{remark}\label{rem:freealgebroids}
By Lemma \ref{lem:actionisfree}, the free dg-Lie algebroid on a $k$-linear map $V\rt T_A$ is given by the action Lie algebroid $A\otimes \mm{Lie}(V)$ associated to the free Lie algebra on $V$ (and similarly in the $L_\infty$-case). The left adjoint to the forgetful functor $\dgLieAlgd_A\rt \dgMod_A/T_A$ is explicitly described in \cite{kap07}.
\end{remark}

\paragraph{Semi-model categories}
Recall that a (left) \emph{semi-model category} \cite{spi01,fre09} is a bicomplete category $\cat{M}$ equipped with wide subcategories of weak equivalences, cofibrations and fibrations, subject to the following conditions:
\begin{enumerate}
\item The weak equivalences have the two out of three property and the weak equivalences, fibrations and cofibrations are stable under retracts.
\item The cofibrations have the left lifting property with respect to the trivial fibrations. The trivial cofibrations \emph{with cofibrant domain} (i.e.\ with a domain $X$ for which the map $\emptyset \rt X$ is a cofibration) have the left lifting property with respect to the fibrations.
\item Every map can be factored functorially into a cofibration, followed by a trivial fibration. Every map \emph{with cofibrant domain} can be factored functorially into a trivial cofibration followed by a fibration.
\item The fibrations and trivial fibrations are stable under transfinite composition, products and base change.
\end{enumerate}
An adjunction $F\colon \cat{M}\leftrightarrows \cat{N}\colon G$ between two semi-model categories is a \emph{Quillen adjunction} if the right adjoint $G$ preserves fibrations and trivial fibrations. It is a Quillen equivalence when a map $X\rt G(Y)$ is a weak equivalence if and only if its adjoint map $F(X)\rt Y$ is a weak equivalence, for any cofibrant $X\in \cat{M}$ and fibrant $Y\in \cat{N}$.

Most model-categorical constructions can be performed in semi-model categories as well. For example, the category of (co)simplicial diagrams in $\cat{M}$ carries a Reedy semi-model structure, which can be used to define (co)simplicial resolutions and simplicial sets of maps in $\cat{M}$. The latter describe the mapping spaces in the $\infty$-categorical localisation $\cat{M}[W^{-1}]$ of $\cat{M}$ at its weak equivalences. For a detailed description of the basic theory of semi-model categories, we refer to \cite{spi01,fre09}.

\begin{definition}
A semi-model category $\cat{M}$ is \emph{tractable} if its underlying category is locally presentable and if there exist sets of maps \emph{with cofibrant domain} $I$ and $J$ with the property that a map has the right lifting property against $I$ (resp.\ $J$) if and only if it is a trivial fibration (resp.\ a fibration). We refer to the maps in $I$ (resp.\ $J$) as the \emph{generating (trivial) cofibrations}. 
\end{definition}
\begin{remark}
The cofibrations and trivial fibrations in a semi-model category $\cat{M}$ determine each other via the lifting property. When $\cat{M}$ is tractable, the fibrations are characterised by the right lifting property against the trivial cofibrations between cofibrant objects. This becomes particularly useful when considering homotopy limits: when $\cat{M}$ is tractable, the diagram category $\Fun(\cat{I}, \cat{M})$ carries a unique tractable semi-model structure whose cofibrations and weak equivalences are given pointwise. In this case, the adjunction
$$\xymatrix{
\Del\colon \cat{M} \ar@<1ex>[r] & \Fun(\cat{I}, \cat{M})\colon \lim\ar@<1ex>[l]
}$$
is a Quillen pair because the left adjoint preserves cofibrations and trivial cofibrations with cofibrant domain. The right derived functor computes the homotopy limit.
\end{remark}
The main purpose of semi-model structures is that they are easier to transfer along adjunctions:
\begin{lemma}[cf.\ {\cite[Proposition 12.1.4]{fre09}}]\label{app:lem:transferlemma}
Let $F\colon \cat{M}\leftrightarrows \cat{N}\colon G$ be an adjunction between locally presentable categories and suppose that $\cat{M}$ carries a tractable semi-model structure with sets of generating (trivial) cofibrations $I$ and $J$.

Define a map in $\cat{N}$ to be a weak equivalence (fibration) if its image under $G$ is a weak equivalence (fibration) in $\cat{M}$ and a cofibration if it has the left lifting property against the trivial fibrations. Assume that the following condition holds:
\begin{itemize}[leftmargin=*]
\item Let $f\colon A\rt B$ be a map in $\cat{N}$ with cofibrant domain, obtained as a transfinite composition of pushouts of maps in $F(J)$. Then $f$ is a weak equivalence.
\end{itemize}
Then the above classes of maps determine a tractable semi-model structure on $\cat{N}$ whose generating (trivial) cofibrations are given by $F(I)$ and $F(J)$.
\end{lemma}
\begin{proof}
The factorisation axioms follow from the small object argument. The only non-trivial thing to check is the lifting axiom for trivial cofibrations between cofibrant objects against fibrations. If $A\rt B$ is a trivial cofibration between cofibrant objects, we can factor it as an iterated pushout $A\rt \tilde{A}$ of maps in $F(J)$, followed by a fibration $\tilde{A}\rt B$. Since the map $A\rt \tilde{A}$ is a weak equivalence, $\tilde{A}\rt B$ is a trivial fibration and the map $A\rt B$ is a retract of the map $A\rt \tilde{A}$. The latter has the lifting property against the fibrations by definition.
\end{proof}
\begin{example}
Let $\cat{M}$ be a tractable semi-model category and let $\cat{I}$ be a small category. Then the category $\mm{Fun}(\cat{I}, \cat{M})$ carries the \emph{projective} semi-model structure, in which a map is a weak equivalence (fibration) if it is a levelwise weak equivalence (fibration) in $\cat{M}$. The functor $\colim\colon \mm{Fun}(\cat{I}, \cat{M})\rt \cat{M}$ is a left Quillen functor, whose left derived functor takes homotopy colimits.
\end{example}
\begin{definition}\label{def:projectivemodelstructure}
A diagram $F\colon \cat{I}\rt \cat{M}$ with values in a tractable semi-model category is \emph{projectively cofibrant} if it is cofibrant for this projective semi-model structure. In particular, this means that $\colim F$ is a model for the homotopy colimit of $F$.
\end{definition}

\section{Main results}\label{sec:main}
In this section we will state our main results and collect some immediate consequences, leaving the proofs to Section \ref{la:sec:technicalities}. First of all, consider the free-forgetful adjunction
$$\xymatrix{
F\colon \dgMod_A/T_A\ar@<1ex>[r] & \dgLieAlgd_A\colon U\ar@<1ex>[l]
}$$
between the category of dg-$A$-modules over $T_A$ and the category of dg-Lie algebroids over $A$. Our first result asserts that the usual projective model structure on dg-$A$-modules can be transferred to a semi-model structure along this adjunction.
\begin{theorem}\label{la:thm:modelstructureondglralgebras}
The category $\dgLieAlgd_A$ of dg-Lie algebroids over $A$ and the category $\dgLooAlgd_A$ of $L_\infty$-algebroids over $A$ both admit a right proper, tractable semi-model structure, in which a map is a weak equivalence (resp.\ a fibration) if and only if it is a quasi-isomorphism (a degreewise surjection).
\end{theorem}
It is \emph{not true} that the quasi-isomorphisms and surjections define a \emph{genuine} model structure on dg-Lie algebroids, as has been asserted in \cite{vez13}, even when the cdga $A$ is free. More precisely, the argument in loc.\ cit.\ relies on Quillen's path object argument, but does not explicitly check the condition that every dg-Lie algebroid admits a fibrant replacement. This condition tends to be easily satisfied, but the following example demonstrates shows that dg-Lie algebroids may \emph{fail} to have a fibrant replacement:
\begin{example}\label{ex:counter}
Let $A=k[x, y]$ and consider the quotient $B=k[x, y]/(x-y)$. Let $\mf{g}$ be the free $A$-linear Lie algebra generated by the $A$-module $B^{\oplus 2}=B\big<e_1, e_2\big>$. Equivalently, $\mf{g}$ is the free $B$-linear Lie algebra on two generators $e_1, e_2$, considered as a Lie algebra over $A$. Suppose that the zero map $\mf{g}\rt T_A$ factors over a fibrant dg-Lie algebroid
$$\xymatrix{
\mf{g}\ar[r]^\iota & \mf{h}\ar@{->>}[r]^-\rho & T_A.
}$$
Being fibrant means that $\mf{h}$ is \emph{transitive}, i.e.\ its anchor map is surjective. We claim that $\iota$ can never be a quasi-isomorphism. To see this, let $v\in  \mf{h}$ be an element in $\mf{h}$ and consider the following two equalities in $\mf{h}$:
\begin{align*}
\big[x\cdot \iota(e_1), [y\cdot \iota(e_2), v]\big] & = xy\cdot \big[\iota(e_1), [\iota(e_2), v]\big] - x\cdot\rho(v)(y)\cdot \iota\big([e_1, e_2]\big)\\
\big[y\cdot \iota(e_1), [x\cdot \iota(e_2), v]\big] & = xy\cdot \big[\iota(e_1), [\iota(e_2), v]\big] - y\cdot\rho(v)(x)\cdot \iota\big([e_1, e_2]\big).
\end{align*}
The left hand sides agree by definition of $\mf{g}$, since $x=y$ in $B$. If we let $v$ be an element such that $\rho(v)=\dau/\dau y$, then it follows that
$$
\iota\Big(x\cdot [e_1, e_2]\Big)=0.
$$
This means that the kernel of $\pi_0(\iota)\colon \pi_0(\mf{g})=\mf{g}\rt \pi_0(\mf{h})$ always contains the (non-zero) element $x\cdot [e_1, e_2]$.
\end{example}
Instead of using the path object argument, our proof of Theorem \ref{la:thm:modelstructureondglralgebras} depends on an analysis of pushouts of generating trivial cofibrations. Such pushouts of dg-Lie algebroids (and $L_\infty$-algebroids) have a similar structure as pushouts of maps between free algebras over an operad. However, some extra care is needed because in the case of Lie algebroids, one can add generators that act non-trivially on $A$. For this reason, we postpone the proof of Theorem \ref{la:thm:modelstructureondglralgebras} to Section \ref{la:sec:technicalities}, where we also prove the following result:
\begin{theorem}\label{la:thm:monadicity}
The forgetful functors 
$$
U\colon \dgLieAlgd_A\rt \dgMod_A/T_A \qquad \text{and} \qquad U\colon \dgLooAlgd_A\rt \dgMod_A/T_A
$$
are right Quillen functors with the following two properties:
\begin{itemize}
\item[(a)] they preserve cofibrant objects, i.e.\ any cofibrant dg-Lie-algebroid is cofibrant as a dg-$A$-module.
\item[(b)] they preserve sifted homotopy colimits. More precisely, let $\cat{J}$ be a non-empty category such that the diagonal $\Del: \cat{J}\rt \cat{J}\times \cat{J}$ is homotopy cofinal and let $\mf{g}\colon \cat{J}\rt \dgLieAlgd_A$ be a projectively cofibrant diagram (Definition \ref{def:projectivemodelstructure}). Then the natural map
$$\xymatrix{
\hocolim_{\cat{J}} U(\mf{g})\ar[r] & U(\colim_{\cat{J}} \mf{g})
}$$
is a weak equivalence of dg-$A$-modules over $T_A$. 
\end{itemize}
\end{theorem}
\begin{example}\label{ex:totalisation}
Suppose that $\mf{g}_\sbullet\colon \Del^{\op}\rt \dgLieAlgd_A$ is a simplicial diagram of dg-Lie algebroids. Theorem \ref{la:thm:monadicity} implies that the homotopy colimit of $\mf{g}_\sbullet$ can be computed as follows: taking normalised chains, we obtain a bicomplex $N(\mf{g})_{p, q}\subseteq (\mf{g}_p)_q$ together with a map $N(\mf{g})\rt T_A$, where $T_A$ is concentrated in bidegrees $(0, q)$. Taking total complexes, we obtain a map of chain complexes $\mm{Tot}(\mf{g}_\bullet)\rt T_A$. Using the Eilenberg-Zilber map, this total complex inherits the structure of a dg-Lie algebroid. This models the homotopy colimit of $\mf{g}_\sbullet$ because the total complex $\mm{Tot}(\mf{g}_\bullet)$ computes its homotopy colimit in the category of chain complexes (over $T_A$).
\end{example}

Theorem \ref{la:thm:monadicity} implies that the free-forgetful adjunction from dg-Lie algebroids to chain complexes (over $T_A$) induces a monadic adjunction of $\infty$-categories. Informally, this means that the forgetful functor still behaves like a forgetful functor when considered from a homotopical perspective. As an application of this result, consider the inclusion
$$\xymatrix{
i\colon \dgLie_A\ar[r] & \dgLieAlgd_A; \hspace{4pt} \mf{g}\ar@{|->}[r] & \big(\mf{g}\rto{0} T_A\big).
}$$
of the category of dg-Lie algebras over $A$ into the category of dg-Lie algebroids.
\begin{proposition}\label{la:prop:monadicityoverliealgebras}
Endow the category $\dgLie_A$ of dg-Lie algebras over $A$ with the model structure transferred from $\dgMod_A$. Then the above inclusion functor is part of a Quillen adjunction
$$\xymatrix{
i\colon \dgLie_A\ar@<1ex>[r] & \dgLieAlgd_A \ar@<1ex>[l]\colon \ker
}$$
whose right adjoint sends a dg-Lie-algebroid to the kernel of its anchor map. The right derived functor $\mathbb{R}\ker$ detects equivalences and preserves all sifted homotopy colimits.
\end{proposition}
\begin{proof}
One easily verifies that the functor $\ker$ is right Quillen and fits into a commuting diagram of right Quillen functors
\begin{equation}\label{la:diag:liealgebrasandalgebroids}\vcenter{\xymatrix{
\dgLieAlgd_A\ar[r]^-{\ker}\ar[d]_U & \dgLie_A\ar[d]^U\\
\dgMod_A/T_A\ar[r]_-{\ker} & \dgMod_A.
}}\end{equation}
Since the vertical forgetful functors detect equivalences and preserve sifted homotopy colimits (see \cite[Proposition 7.8]{pav14} for the case of Lie algebras), it suffices to check that the right derived functor of $\ker\colon \dgMod_A/T_A\rt \dgMod_A$ has these properties as well. But it follows immediately from the fact that $\dgMod_A$ is a stable model category that taking homotopy pullbacks along $0\rt T_A$ detects equivalences and preserves all homotopy colimits indexed by contractible categories.
\end{proof}
\begin{remark}
The above proposition asserts that the $\infty$-category of Lie algebroids over $A$ is monadic over the $\infty$-category of Lie algebras over $A$. In particular, even though the functor $\dgLie_A\rt \dgLieAlgd_A$ is fully faithful, its derived functor is \emph{not} fully faithful; the derived counit map is given at the level of $A$-modules by a map $\mf{g}\oplus T_A[-1]\rt \mf{g}$. 

In \cite{gai16}, Lie algebroids have also been described as algebras for a certain monad on the $\infty$-category $\cat{Lie}_A$ of $A$-linear Lie algebras. However, the monad used in loc.\ cit.\ is constructed in a rather indirect way and is not described explicitly in algebraic terms.
\end{remark}
\begin{corollary}\label{la:cor:resolutionbyliealgebras}
Any dg-Lie algebroid $\mf{g}$ arises as the homotopy colimit of a diagram $\mf{h}_\sbullet\colon \Del^{\op}\rt \dgLieAlgd_A$ where each $\mf{h}_n$ is weakly equivalent to a Lie algebra over $A$.
\end{corollary}
\begin{proof}
Use the bar resolution associated to the Quillen pair $\dgLie_A\leftrightarrows \dgLieAlgd_A$ \cite{blu14} or associated to the induced monadic adjunction of $\infty$-categories \cite[Proposition 4.7.4.14]{lur16}.
\end{proof}
\begin{remark}
Let $0\stackrel{\sim}{\longrightarrow} P(T_A)\twoheadrightarrow T_A$ be a factorisation of the zero map into a trivial cofibration, followed by a fibration. For any Lie algebroid $\mf{g}$, a Lie algebroid map $\mf{g}\rt P(T_A)$ determines a null-homotopy (by Lie algebroid maps) of the anchor map $\mf{g}\rt T_A$. Because the semi-model structure on $\dgLieAlgd_A$ is right proper, pulling back along $0\rt P(T_A)$ determines a right Quillen equivalence
$$\xymatrix{
\dgLieAlgd_A/P(T_A)\ar[r]^-\sim & \dgLieAlgd_A/0\simeq \dgLie_A.
}$$
In other words, the model category of $A$-linear dg-Lie algebras is Quillen equivalent to the semi-model category of dg-Lie algebroids endowed with a null-homotopy (by Lie algebroid maps) $\mf{g}\rt P(T_A)$ of their anchor map.
\end{remark}

\begin{corollary}\label{la:cor:equivalencetohomotopyLRalgebras}
The inclusion $j\colon \dgLieAlgd_A\rt \dgLooAlgd_A$ is the right adjoint of a Quillen equivalence.
\end{corollary}
In particular, any $L_\infty$-algebroid $\mf{g}$ is quasi-isomorphic to a dg-Lie algebroid. However, note that this dg-Lie algebroid is typically much \emph{larger} than $\mf{g}$ itself: it is obtained by first taking a cofibrant replacement of $\mf{g}$ and then applying the left adjoint to $j$.
\begin{proof}
The functor $j$ fits into a commuting diagram of right Quillen functors
$$\xymatrix{
\dgLieAlgd_A\ar[r]^j \ar[d]_{\ker} & \dgLooAlgd_A\ar[d]^{\ker}\\
\dgLie_A \ar[r]_-{w^*} & \cat{L_\infty Alg}^{\dg}_A.
}$$
Here $w^*$ is the forgetful functor associated to the map of operads $w\colon L_\infty\rt \ope{Lie}$. This functor is part of a Quillen equivalence
$$\xymatrix{
w_!\colon \dgLie_A\ar@<1ex>[r] & \cat{L_\infty Alg}^{\dg}_A\colon w^*\ar@<1ex>[l].
}$$
because $w$ is a weak equivalence between $\Sigma$-cofibrant operads. Since $w^*$ and the vertical functors have right derived functors that detect equivalences and preserve all sifted homotopy colimits, it follows that $j$ has these properties as well. 

Let $L$ be the left adjoint to the right Quillen functor $j$. Because $j$ detects weak equivalences, it suffices to show that the (derived) unit map $\eta\colon \mf{g}\rt jL(\mf{g})$ is a weak equivalence for each cofibrant $L_\infty$-algebroid $\mf{g}$. By Corollary \ref{la:cor:resolutionbyliealgebras} (in the $L_\infty$-case), $\mf{g}$ is weakly equivalent to the homotopy colimit of a simplicial diagram of $L_\infty$-algebroids, each of which is weakly equivalent to an $L_\infty$-algebra. Because $j$ and $L$ preserve sifted homotopy colimits, it suffices to show that $\eta$ is a weak equivalence when $\mf{g}$ is a cofibrant $L_\infty$-algebra over $A$. But in that case, the (derived) unit map $\mf{g}\rt jL(\mf{g})$ agrees with the (derived) unit map $\mf{g}\rt w^*w_!\mf{g}$, which is a weak equivalence.
\end{proof}
Finally, let us mention that the proofs of Theorem \ref{la:thm:modelstructureondglralgebras} and Theorem \ref{la:thm:monadicity} also apply in various other situations where the notion of a dg-Lie algebroid makes sense:
\begin{variant}\label{var:gradedmixed}
Let $\cat{M}_k$ be the category of \emph{graded-mixed complexes} over $k$, i.e.\ $\mathbb{Z}$-graded chain complexes $\{V(p)\}_{p\in \mathbb{Z}}$ equipped with maps
$$\xymatrix{
d\colon V(p)\ar[r] & V(p+1)[-1]
}$$
such that $d^2=0$. Recall from \cite{pan13} that there is a cofibrantly generated model structure on $\mc{M}_k$, with weak equivalences (fibrations) given by the degreewise quasi-isomorphisms (surjections). This is a symmetric monoidal model structure for the tensor product 
$$
(V\otimes W)(p)=\textstyle{\bigoplus}_q V(q)\otimes_k W(p-q) \qquad \qquad d_{V\otimes W} = d_V\otimes 1 + 1\otimes d_W.
$$
Interpreting Definition \ref{la:def:lralgebra} in $\mc{M}_k$ instead of chain complexes, we obtain a notion of \emph{graded-mixed dg-Lie algebroid} $\mf{g}\rt T_A$ over a cdga $A$. Here $A$ and $T_A$ are considered as graded mixed complexes of weight zero. The proof of Theorem \ref{la:thm:modelstructureondglralgebras} shows that the category of such graded-mixed dg-Lie algebroids carries a transferred semi-model structure.
\end{variant}
\begin{variant}\label{var:tame}
The category $\dgMod_A$ can be endowed with the \emph{contraderived}, or \emph{tame} model structure, of which the projective model structure is a right Bousfield localisation. In this model structure, the fibrations are the surjections and the cofibrations are the monomorphisms whose cokernels are projective as graded $A$-modules. 

The category of dg-Lie algebroids over $A$ can then be endowed with a semi-model structure in which the fibrations are the surjections and the weak equivalences are the maps that induce a tame weak equivalence on the underlying dg-$A$-modules. Furthermore, the forgetful functor 
$$
\dgLieAlgd_A\rt \dgMod_A/T_A
$$
preserves cofibrant objects and sifted homotopy colimits. Both assertions are proven in exactly the same way as the above two theorems, using Proposition \ref{la:prop:reductionto0map} (see also Remark \ref{rem:fortame}).
\end{variant}
\begin{variant}\label{var:presheaves}
Let $\cat{C}$ be a site endowed with a presheaf of commutative dg-algebras
$$\xymatrix{
\mc{O}\colon \cat{C}^\op\ar[r] & \dgCAlg_k.
}$$
There is a natural presheaf $T_\mc{O}$ whose value on $c\in\cat{C}$ is the dg-Lie algebra (and dg-$\mc{O}(c)$-module) of natural derivations of $\mc{O}$ over $\cat{C}/c$. A \emph{presheaf of dg-Lie algebroids} over $\mc{O}$ is a map $\mf{g}\rt T_\mc{O}$ of presheaves of dg-$\mc{O}$-modules and dg-Lie algebras, satisfying the conditions of Definition \ref{la:def:lralgebra}. 

There is a semi-model structure on the category of presheaves of dg-Lie algebroids over $\mc{O}$ whose weak equivalences are the maps inducing isomorphisms on homology sheaves. Indeed, one can apply the proof of Theorem \ref{la:thm:modelstructureondglralgebras} to obtain a transferred semi-model structure on presheaves of dg-Lie algebroids over $\mc{O}$, starting with the following model structure on presheaves of chain complexes:
\begin{lemma}
Let $\cat{C}$ be a site, $k$ a field and $\mm{PSh}(\cat{C})^{\dg}$ the category of (unbounded) complexes of presheaves of $k$-vector spaces over $\cat{C}$. This carries a cofibrantly generated model structure, in which the cofibrations are the monomorphisms and the weak equivalences are local quasi-isomorphisms, i.e.\ maps inducing isomorphisms on homology sheaves. Furthermore, this model structure is monoidal model with respect to the pointwise tensor product $(F\otimes G)(c) = F(c)\otimes_k G(c)$.
\end{lemma}
\begin{proof}
Let us first consider the category $\mm{Sh}^\dg(\cat{C})$ of (unbounded) complexes of \emph{sheaves} of $k$-vector spaces. Since the category $\mc{G}$ of sheaves of $k$-vector spaces is a Grothendieck abelian category, the category of complexes in $\mc{G}$ carries a combinatorial model structure, whose cofibrations are the monomorphisms and weak equivalences are the local quasi-isomorphisms (see e.g.\ \cite[Corollary 7.1]{gil07}). 

The category $\mc{G}$ carries a tensor product $\tilde{\otimes}$, obtained by taking the associated sheaf of the pointwise tensor product. Note that $F\tilde{\otimes}(-)$ is exact for every sheaf $F\in\mc{G}$: indeed, the pointwise tensor product of $k$-vector spaces is exact, as is taking associated sheaves. This implies that $\mm{Sh}^\dg(\cat{C})$ forms a monoidal model category \cite[Theorem 5.1]{gil07}.

For the presheaf case, we apply Smith's recognition theorem \cite[Proposition A.2.6.8]{lur09}. Note that:
\begin{enumerate}[leftmargin=*]
\item The monomorphisms in $\mm{PSh}^\dg(\cat{C})$ form a weakly saturated class generated by a set of morphisms \cite[Lemma A.2.8.3]{lur09}. 
\item The injective local quasi-isomorphisms form a weakly saturated class. Indeed, it is the intersection of the weakly saturated class of monomorphisms and the inverse image of the trivial cofibrations in $\mm{Sh}^\dg(\cat{C})$ under the associated sheaf functor $a\colon \mm{PSh}^\dg(\cat{C})\rt \mm{Sh}^\dg(\cat{C})$.
\item Because $a$ is exact, a map $F\rt G$ is a local quasi-isomorphism if and only if $a(F)\rt a(G)$ is. The local quasi-isomorphisms are thus the inverse image under $a$ of the weak equivalences in $\mm{Sh}^\dg(\cat{C})$, and hence form an accessibly embedded accessible subcategory \cite[Corollary A.2.6.5, A.2.6.6]{lur09}.
\item The local quasi-isomorphisms have the two-out-of-three property.
\item If a map has the right lifting property against all monomorphisms, then it is in particular a pointwise trivial fibration of complexes and hence a local quasi-isomorphism.
\end{enumerate} 
It follows that $\mm{PSh}^\dg(\cat{C})$ has the required combinatorial model structure. To see that it is monoidal model, note that the pushout-product $f\Box g$ is a monomorphism when $f$ and $g$ are monomorphisms: this can be verified pointwise, where it follows from the case of chain complexes over a field (where the injective and projective model structure coincide). If $f$ is furthermore a local quasi-isomorphism, then $f\Box g$ is a local quasi-isomorphism: its image $a(f\Box g)$ under the associated sheaf functor agrees with the pushout-product of $a(f)$ and $a(g)$ in $\mm{Sh}^{\dg}(\cat{C})$, which is a local quasi-isomorphism.
\end{proof}
\end{variant}

\section{Filtrations on cell attachments}\label{la:sec:technicalities}
This section is devoted to the proofs of Theorem \ref{la:thm:modelstructureondglralgebras} and Theorem \ref{la:thm:monadicity}. Just as in the case of algebras over an operad, these proofs rely on an analysis of the pushout of a diagram of dg-Lie algebroids (or $L_\infty$-algebroids) of the form
\begin{equation}\label{diag:pushoutofliealgebroids}\vcenter{\xymatrix@R=1.8pc@C=1.9pc{
F(V)\ar[r]^-{F(i)}\ar[d] & F(W)\ar[d]\\
\mf{g}\ar[r] & \mf{h}.
}}\end{equation}
We will show that the map $\mf{g}\rt \mf{h}$ can be decomposed into a sequence of maps $\mf{g}^{(p)}\rt \mf{g}^{(p+1)}$, whose associated graded is controlled by the \emph{reduced enveloping operad} of the dg-Lie algebroid $\mf{g}$. The difference from the case of algebras over operads is that the maps $\mf{g}^{(p)}\rt \mf{g}^{(p+1)}$ need not be injective in general.

Throughout, we will only treat $L_\infty$-algebroids; the case of dg-Lie algebroids proceeds in exactly the same manner, replacing all appearances of the $L_\infty$-operad by the Lie operad.

\paragraph{Filtrations}
Let $\Mod_k^{\dg, \mathbb{N}}$ be the category of sequences of chain complexes 
\begin{equation}\label{la:diag:filteredcomplex}\smash{\xymatrix{
V^{(0)}\ar[r] & V^{(1)}\ar[r] & V^{(2)}\ar[r] &\dots
}}\end{equation}
endowed with the Reedy model structure. We will refer to an object $V$ of $\Mod_k^{\dg, \mathbb{N}}$ as a \emph{weakly filtered} chain complex. An object is Reedy cofibrant if and only if \eqref{la:diag:filteredcomplex} consists of monomorphisms, in which case it can be interpreted as a genuine filtration on $\colim V$. We will say that an element in $V^{(p)}$ is of \emph{weight} $\leq p$ and an element of $V^{(p)}/V^{(p-1)}$ is of \emph{weight} $p$. Degrees always indicate homological degrees.

The category of weakly filtered chain complexes has a closed symmetric monoidal structure, given by 
$$
(V\otimes W)^{(n)} = \colim V^{(p)}\otimes W^{(q)}.
$$
The colimit is taken over the full subcategory of $(p, q)\in \mathbb{N}\times \mathbb{N}$ for which $p+q\leq n$. The symmetry isomorphism given by the symmetry isomorphisms of chain complexes $V^{(p)}\otimes W^{(q)}\rt W^{(q)}\otimes V^{(p)}$, i.e.\ there are no extra signs depending on $p$ and $q$. 

There are two Quillen pairs
$$\xymatrix{
\colim\colon \Mod_k^{\dg, \mathbb{N}} \ar@<1ex>[r] & \dgMod_k\colon i\ar@<1ex>[l] & \mm{gr}\colon \Mod_k^{\dg, \mathbb{N}} \ar@<1ex>[r] & \Mod_k^{\dg, \mm{gr}}\ar@<1ex>[l]\colon j.
}$$
Here $i$ sends a chain complex to the constant diagram on $V$ (and will be omitted from the notation) and `$\mm{gr}$' sends a sequence $V$ to the $\mathbb{N}$-graded chain complex $V^{(\bullet)}/V^{(\bullet-1)}$, with right adjoint sending a graded chain complex $W$ to the sequence consisting of zero maps. Each of the above functors is symmetric monoidal. 
\begin{remark}\label{la:rem:associatedgradeddetectsequivalences}
The functor $\mm{gr}\colon \Mod_k^{\dg, \mathbb{N}}\rt \Mod_k^{\dg, \mm{gr}}$ detects weak equivalences between cofibrant objects: this is just the fact that weak equivalences of filtered chain complexes are detected on the associated graded.
\end{remark}
The notions of $L_\infty$-algebras and $L_\infty$-algebroids over $A$ have obvious weakly filtered and graded analogues ($A$ is always of weight $\leq 0$). For example, a weakly filtered $L_\infty$-algebroid over $A$ is an object $\mf{g}$ in $\Mod_k^{\dg, \mathbb{N}}$ together with  
\begin{enumerate}
 \item the structure of an $A$-module, i.e.\ natural chain maps $A\otimes \mf{g}^{(i)}\rt \mf{g}^{(i)}$.
 \item an $L_\infty$-algebra structure in $\Mod_k^{\dg, \mathbb{N}}$, i.e.\ for each $p\geq 0$ a matching family of $n$-ary maps $[-, \cdots, -]\colon \mf{g}^{(i_1)}\otimes \cdots\otimes \mf{g}^{(i_n)}\rt \mf{g}^{(p)}$, for all $i_1+\cdots+i_n\leq p$.
 \item a map $\mf{g}\rt T_A$ of $L_\infty$-algebras and $A$-modules in $\Mod_k^{\dg, \mathbb{N}}$, where $T_A$ is of weight $\leq 0$.
\end{enumerate}
When $\mf{g}$ is Reedy cofibrant (i.e.\ a filtered chain complex), this is simply the structure of an $L_\infty$-algebroid on $\colim(\mf{g})$ whose entire structure respects the filtration. Let us denote the categories of weakly filtered and graded $L_\infty$-algebroids over $A$ by
$$
\LooAlgd_A^{\dg, \mathbb{N}} \qquad \text{and} \qquad \LooAlgd_A^{\dg, \mm{gr}}.
$$ 
The description of the free $L_\infty$-algebroid on a chain complex over $T_A$ (see Remark \ref{rem:freealgebroids}) also applies to the weakly filtered and graded settings: one first takes the free (weakly filtered, graded) $L_\infty$-algebra over $T_A$ and then takes the associated action $L_\infty$-algebroid (Example \ref{la:ex:actionalgebroids}). This yields a commuting diagram of left adjoints
$$\xymatrix{
\Mod_k^{\dg, \mathbb{N}}/T_A\ar[d]_{\mm{Free}} & \dgMod_k/T_A\ar[d]_{\mm{Free}}\ar[l]_-i & \Mod_k^{\dg, \mathbb{N}}/T_A \ar[l]_{\colim} \ar[d]^{\mm{Free}}\ar[r]^{\mm{gr}} & \Mod_k^{\dg, \mm{gr}}/T_A\ar[d]^{\mm{Free}}\\
\LooAlgd^{\dg, \mathbb{N}} & \dgLooAlgd_A  \ar[l]^-i & \LooAlgd_A^{\dg, \mathbb{N}} \ar[r]_{\mm{gr}}\ar[l]^{\colim} & \LooAlgd^{\dg, \mm{gr}}.
}$$
The vertical functors are the free functors, sending a (weakly filtered, graded) chain complex $V$ over $T_A$ to the action $L_\infty$-algebroid $A\otimes L_\infty(V)$ associated to the free $L_\infty$-algebra on $V$. All horizontal functors can be computed at the level of chain complexes. For example, the colimit of a weakly filtered $L_\infty$-algebroid is simply the colimit of the underlying sequence of chain complexes, together with a certain $L_\infty$-algebroid structure on it. Note that the inclusion functor $i$ is both a right adjoint (to the functor `$\colim$') and a left adjoint (to the functor taking the weight $\leq 0$ part).

\paragraph{Coproducts with $L_\infty$-algebras}
In this section we will study the simplest type of pushout diagram \eqref{diag:pushoutofliealgebroids}: the case of a coproduct of a (weakly filtered) $L_\infty$-algebroid $\mf{g}$ over $A$ with the free $L_\infty$-algebroid generated by a (weakly filtered) dg-$A$-module $V$, equipped with the \emph{zero map} to $T_A$. 

Such coproducts are much easier to describe than coproducts for non-zero maps $V\rt T_A$. Indeed, the coproduct $\mf{g}\amalg \smash{F\big(V\rto{0} T_A\big)}$ fits into a retract diagram
$$\xymatrix{
\mf{g}=\mf{g}\amalg F(0)\ar[r] & \mf{g}\amalg F\big(V\rto{0} T_A\big)\ar[r] & \mf{g}.
}$$
This construction is the left adjoint in an adjunction
\begin{equation}\label{diag:freeretractiveliealgebroid}\vcenter{\xymatrix{
\mf{g}\amalg F\big((-)\rto{0} T_A\big)\colon \Mod^{\dg, \mathbb{N}}_A \ar@<1ex>[r] & \mf{g}/\LooAlgd^{\dg, \mathbb{N}}/\mf{g}\colon \ker\ar@<1ex>[l]
}}\end{equation}
where the right adjoint sends a retract diagram $\mf{g}\rt \mf{h}\rt \mf{g}$ to the kernel of $\mf{h}\rt \mf{g}$.

The category of retract diagrams of (weakly filtered) $L_\infty$-algebroids
$$\xymatrix{
\mf{g}\ar[r] & \mf{h}=\mf{g}\oplus \mf{m}\ar[r] & \mf{g}
}$$ 
can be identified with the category of algebras over an operad in (weakly filtered) chain complexes over $k$. Indeed, such a retract diagram can equivalently be encoded by the following kind of algebraic structure on $\mf{m}$:
\begin{itemize}[leftmargin=*]
\item[$\sbullet$] $\mf{m}$ has the structure of a (weakly filtered) $A$-module.
\item[$\sbullet$] $\mf{m}$ comes equipped with an $A$-linear $L_\infty$-structure, since the anchor map vanishes on $\mf{m}$.
\item[$\sbullet$] for each set of elements $x_1, \dots, x_n\in \mf{g}$, the $(n+k)$-ary bracket on $\mf{g}\oplus \mf{m}$ determines a $k$-ary operation $[x_1, \dots, x_n, (-)]\colon \mf{m}^{\otimes k}\rt \mf{m}$ of degree $k-2$, for each $k\geq 1$. 
\end{itemize}
These operations have to satisfy equations stating that certain sums of their composites are zero. This type of algebraic structure can precisely be encoded by means of an operad, which has no nullary operations (as one sees from the above description).
\begin{definition}\label{la:def:reducedenvelopingoperad}
The \emph{reduced enveloping operad} $\rEnv_\mf{g}$ of a weakly filtered $L_\infty$-algebroid $\mf{g}$ is the (reduced) weakly filtered dg-operad over $k$ whose algebras $\mf{m}$ are retract diagrams of $L_\infty$-algebroids $\mf{g}\rt \mf{h}=\mf{g}\oplus \mf{m}\rt \mf{g}$.
\end{definition}
\begin{remark}\label{la:rem:presentationofreducedenveloping}
The above definition is somewhat imprecise. More accurately, one can construct the operad $\rEnv_\mf{g}$ in terms of generators of the form 
\begin{itemize} \setlength{\itemsep}{0pt}\setlength{\parskip}{0pt}\setlength{\parsep}{0pt} 
\item $\mu_a$ for $a\in A$ (left multiplication by $a$)
\item $[-, \dots, -]$ (the $L_\infty$-structure on $\mf{m}$)
\item $[x_1, \dots, x_n, -, \dots, -]$ for elements $x_{1}, \dots, x_n$ in $\mf{g}$.
\end{itemize}
These generators have to satisfy an obvious list of equations. For example, there are equations expressing the anti-symmetry and Jacobi identities for the various brackets. Furthermore, the brackets $[x_1, \dots, x_n, -, \dots, -]$ depend $A$-multilinearly on the elements $\xi_i$ and are almost all $A$-multilinear operations themselves, viz.\
\begin{align*}
[a\cdot x_1, \dots, x_n, -, \dots, -] &= \mu_a\circ [x_1, \dots, x_n, -, \dots, -]\\
[x, -]\circ \mu_a &= \mu_a\circ [x, -] + \mu_{x(a)}\\
[x_1, \dots, x_n, -, \dots, -]\circ_i \mu_a &= \mu_a\circ [x_1, \dots, x_n, -, \dots, -].
\end{align*}
\end{remark}
\begin{example}
Suppose that $\mf{g}$ is an $A$-linear $L_\infty$-algebra. Then the reduced enveloping operad of $\mf{g}$ is simply the arity $\geq 1$ part of the usual enveloping operad of $\mf{g}$ (as discussed e.g.\ in \cite[Definition 1.5]{ber09}).
\end{example}
\begin{remark}\label{la:rem:functorialityofreducedenveloping}
A map of $L_\infty$-algebroids $f\colon \mf{g}\rt \mf{h}$ induces a map of reduced enveloping operads $f\colon \rEnv_\mf{g}\rt \rEnv_\mf{h}$, which sends each generator $[x_1, \dots, x_n, -, \dots, -]$ to the generator $[f(x_1), \dots, f(x_n), -, \dots, -]$. The corresponding restriction functor between categories of algebras can be identified with the functor
$$\xymatrix{
f^*\colon \mf{h}\big/\dgLooAlgd_A\big/\mf{h}\ar[r] & \mf{g}\big/\dgLooAlgd_A\big/\mf{g}
}$$
sending $\mf{h}\rt \mf{h}\oplus \mf{m}\rt \mf{h}$ to the pullback $\mf{g}\rt (\mf{h}\oplus\mf{m})\times_{\mf{h}} \mf{g}\rt \mf{g}$.
\end{remark}
The operad structure on $\rEnv_\mf{g}$ is not $A$-linear, but there is a canonical map of operads $\mu\colon A\rt \rEnv_\mf{g}$. Here we consider $A$ as an operad with only unary operations. The adjunction \eqref{diag:freeretractiveliealgebroid} can be identified with the adjunction that restricts and induces operadic algebras along $\mu$. In particular, for every (weakly filtered) \emph{dg-$A$-module} $V$, we can identify
\begin{align*}
\mf{g}\coprod F\big(V\stackrel{0}{\rt} T_A\big) &\cong \mf{g}\oplus \Big(\rEnv_\mf{g}\circ_A V\Big)\\
&= \mf{g}\oplus \bigoplus_{p\geq 1} \rEnv_\mf{g}(p)\otimes_{\Sigma_p\ltimes A^{\otimes p}} V^{\otimes p}.
\end{align*}
Here $\circ_A$ denotes the relative composition product over $A$. In exactly the same way, the coproduct of $\mf{g}$ with the free $L_\infty$-algebroid on a map $0\colon V\rt T_A$ of \emph{chain complexes over $k$} can be identified with the composition product
$$
\mf{g}\coprod \mm{Free}\big(V\stackrel{0}{\rt} T_A\big) \cong\mf{g}\coprod F\big(A\otimes V\stackrel{0}{\rt} T_A\big) \cong \mf{g}\oplus \Big(\rEnv_\mf{g}\circ V\Big)
$$
To simplify the above formulas, let us make the following definition:
\begin{definition}\label{la:def:envelopingsequence}
For any (weakly filtered) $L_\infty$-algebroid $\mf{g}$, let $\Env_\mf{g}$ be the symmetric sequence of (weakly filtered) chain complexes over $k$ given by $\Env_\mf{g}(0)=\mf{g}$ and $\Env_\mf{g}(p)=\rEnv_\mf{g}(p)$ for $p\geq 1$. This determines a functor
$$\xymatrix{
\Env\colon \LooAlgd^{\dg, \mathbb{N}}_A\ar[r] & \cat{BiMod}_A^{\dg, \Sigma, \mathbb{N}}
}$$
to the category of $A$-bimodules of (weakly filtered) symmetric sequences. In other words, each $\Env_\mf{g}(p)$ has a commuting left $A$-module and right $\Sigma_p$-equivariant $A^{\otimes p}$-module structure.
\end{definition}
\begin{remark}
The symmetric sequence $\Env_\mf{g}$ has no natural operad structure.
\end{remark}
\begin{remark}\label{la:rem:envelopingforfilteredvsunfiltered}
Let $\mf{g}$ be a weakly filtered $L_\infty$-algebroid of weight $\leq 0$, i.e.\ an ordinary $L_\infty$-algebroid. Then $\Env_\mf{g}$ is of weight $\leq 0$ as well. Similarly, if $\mf{g}$ is a graded $L_\infty$-algebroid, then $\Env_{\mf{g}}$ is a symmetric sequence of graded complexes. In other words, there is a commuting diagram
$$\xymatrix{
\dgLooAlgd_A\ar[d]_{\Env} \ar[r] & \LooAlgd_A^{\dg, \mathbb{N}}\ar[d]_{\Env} & \LooAlgd_A^{\dg, \mm{gr}}\ar[d]^{\Env}\ar[l]\\
\cat{BiMod}_A^{\dg, \Sigma}\ar[r] & \cat{BiMod}_A^{\dg, \Sigma, \mathbb{N}} & \cat{BiMod}_A^{\dg, \Sigma, \mm{gr}}\ar[l]
}$$
where the horizontal functors are the obvious inclusions.
\end{remark}
\begin{construction}\label{cons:symdec}
If $X$ is a symmetric sequence and $p\geq 0$, consider the symmetric sequence
$$
X(p+(-)):= \big\{X(p+q)\big\}_{q\geq 0}.
$$
The symmetric group actions are obtained by restriction along the inclusion $\Sigma_q\rt \Sigma_p\times\Sigma_q\rt \Sigma_{p+q}$. In particular, each $X(p+q)$ carries an action of $\Sigma_p$ commuting with the $\Sigma_q$-action, so that we can consider $X(p+(-))$ as a $\Sigma_p$-object in symmetric sequences. Consequently, any composition product $X(p+(-))\circ Y$ has a natural $\Sigma_p$-action.
\end{construction}
\begin{lemma}\label{la:lem:propertiesofenveloping}
The functor $\Env\colon \LooAlgd_A^{\dg, \mathbb{N}}\rt \cat{BiMod}_A^{\dg, \Sigma, \mathbb{N}}$ has the following properties:
\begin{enumerate}\setlength{\itemsep}{0pt}\setlength{\parskip}{0pt}\setlength{\parsep}{0pt}
 \item It preserves all filtered colimits and reflexive coequalizers.
 \item Let $\mf{g}\rt T_A$ be a map of weakly filtered $k$-linear $L_\infty$-algebras and let $A\otimes\mf{g}\rt T_A$ be the associated action $L_\infty$-algebroid. Then there is a natural isomorphism of symmetric $A$-bimodules
 $$
 \Env_{A\otimes\mf{g}}\cong A\otimes (L_\infty)_\mf{g}
 $$
 where $(L_\infty)_\mf{g}$ is the enveloping operad of the $L_\infty$-algebra $\mf{g}$. The bimodule structure is induced by the canonical bimodule structure on $A$.
 \item If $\mf{g}$ is a weakly filtered $L_\infty$-algebroid and $V$ is a weakly filtered dg-$A$-module, then there is an isomorphism of $A$-bimodules with a $\Sigma_p$-action
 $$
 \Env_{\mf{g}\coprod F\big(V\rto{0} T_A\big)}(p)\cong \Env_\mf{g}\big(p+(-)\big)\circ_A V.
 $$
 \item The functor $\Env$ commutes with taking the colimit and associated graded of a weakly filtered $L_\infty$-algebroid. In other words, there is a commuting diagram
 $$\xymatrix{
\dgLooAlgd_A\ar[d]_{\Env} & \LooAlgd_A^{\dg, \mathbb{N}}\ar[d]_{\Env}\ar[l]_-{\colim}\ar[r]^{\mm{gr}} & \LooAlgd_A^{\dg, \mm{gr}}\ar[d]^{\Env}\\
\cat{BiMod}_A^{\dg, \Sigma} & \cat{BiMod}_A^{\dg, \Sigma, \mathbb{N}} \ar[l]^-{\colim}\ar[r]_{\mm{gr}} & \cat{BiMod}_A^{\dg, \Sigma, \mm{gr}}.
}$$
\end{enumerate}
\end{lemma}
\begin{remark}\label{rem:envelopingoperadfreelooalgebra}
Consider Lemma \ref{la:lem:propertiesofenveloping}(3) in the special case where $A=k$ and $\mf{g}=0$. In this case $\Env_\mf{g}=L_\infty$ is just the $L_\infty$-operad and one retrieves the formula for the enveloping operad of a free $L_\infty$-algebra (cf.\ \cite[Proposition 1.6]{ber09})
$$
\big(L_\infty\big)_{L_\infty(V)}(p)= \sum_{q\geq 0} L_{\infty}(p+q)\otimes_{\Sigma_q} V^{\otimes q}.
$$
\end{remark}
\begin{proof}
Since filtered colimits and reflexive coequalizers of $L_\infty$-algebroids are computed at the level of the underlying complexes, part (1) follows either from the explicit description of $\Env_\mf{g}$ in terms of generators and relations (Remark \ref{la:rem:presentationofreducedenveloping}) or from the fact that for any such diagram $\mf{g}_\sbullet$, there is an isomorphism
\begin{align*}
\Env_{\colim(\mf{g}_\sbullet)} \circ_A V & \cong \big(\colim \mf{g}_\sbullet\big)\amalg F\big(V\rto{0} T_A\big)\\
&\cong \colim \Big(\mf{g}_\sbullet\amalg F\big(V\rto{0} T_A\big)\Big) \cong \colim \big(\Env_{\mf{g}_\sbullet}\big)\circ_A V.
\end{align*}
For (2), consider an action $L_\infty$-algebroid $A\otimes \mf{g}$ and let $V$ be a chain complex. The free $L_\infty$-algebroid on $0\colon V\rt T_A$ is the $A$-linear extension of the free $L_\infty$-algebra $L_\infty(V)$ generated by $V$. It follows that
\begin{align*}
\Env_{A\otimes\mf{g}}\circ V &\cong\big(A\otimes\mf{g}\big)\amalg \mm{Free}\big(V\rto{0} T_A\big) \\
&\cong A\otimes\Big(\mf{g}\amalg L_\infty(V)\Big) \cong \big(A\otimes (L_\infty)_\mf{g}\big)\circ V
\end{align*}
is the $A$-linear extension of the coproduct of $k$-linear $L_\infty$-algebras $\mf{g}\amalg L_\infty(V)$. This induces an isomorphism of symmetric sequences $\Env_{A\otimes \mf{g}}\cong A\otimes (L_\infty)_\mf{g}$, which identifies the $A$-bimodule structure on $\Env_{A\otimes \mf{g}}$ with the bimodule structure on $A\otimes (L_\infty)_\mf{g}$ arising from $A$.

For (3), observe that for any dg-$A$-module $W$, there are natural isomorphisms
\begin{align*}
\Env_{\mf{g}\amalg F(V\rto{0} T_A)}\circ_A W & \cong \mf{g}\amalg F(V\rto{0} T_A)\amalg F(W\rto{0} T_A)\\
&\cong \mf{g}\amalg F(V\oplus W\rto{0} T_A)\\
&\cong \Env_\mf{g}\circ_A (V\oplus W)\\
&\cong \bigoplus_{p\geq 0} \bigg(\bigoplus_{q\geq 0}\Env_\mf{g}(p+q)\otimes_{\Sigma_q\ltimes A^{\otimes q}} V^{\otimes q}\bigg)\otimes_{\Sigma_p\ltimes A^{\otimes p}} W^{\otimes p}\\
&\cong \bigoplus_{p\geq 0}\Big(\Env_\mf{g}(p+(-))\circ_A V\Big)\otimes_{\Sigma_p\ltimes A^{\otimes p}} W^{\otimes p}
\end{align*}
The fourth equality uses the binomial formula $(V\oplus W)^{\otimes n} \cong \bigoplus_{p+q=n} V^{\otimes q}\otimes W^{\otimes p}$. It follows that $\Env_{\mf{g}\amalg F(V)}$ is isomorphic, in each arity $p$, to the $\Sigma_p$-equivariant object $\Env_\mf{g}(p+(-))\circ_A V$ (Construction \ref{cons:symdec}).

For (4), recall that $\Env$ commutes with the inclusions of objects of weight $\leq 0$ (resp.\ graded objects) into weakly filtered objects (Remark \ref{la:rem:envelopingforfilteredvsunfiltered}). This implies that there is a natural transformation 
$$\xymatrix{
\nu\colon \colim \circ \Env\ar[r] & \Env\circ \colim
}$$
and similarly for the functor taking the associated graded. Any weakly filtered $L_\infty$-algebroid can be obtained as a reflexive coequalizer of free $L_\infty$-algebroids generated by weakly filtered chain complexes over $T_A$. Since the functors $\colim$ and $\Env$ preserve reflexive coequalizers, it suffices to check that $\nu$ induces an isomorphism for such a free $L_\infty$-algebroid. 

For a free $L_\infty$-algebroid $\mm{Free}(\rho\colon V\rt T_A)$, we can use part (2) and the description of the enveloping operad of a free $L_\infty$-algebra (Remark \ref{rem:envelopingoperadfreelooalgebra}) to see that 
$$\xymatrix{
\nu\colon \colim\Big(\Env_{\mm{Free}(V)}(p)\Big)\ar[r] & \Env_{\colim(\mm{Free}(V))}(p)
}$$
is the $A$-linear extension of the map
$$\xymatrix@C=1pc{
\colim \left(\bigoplus_{q\geq 0} L_\infty(p+q)\otimes_{\Sigma_q} V^{\otimes q}\right)\ar[r] & \bigoplus_{q\geq 0}L_\infty(p+q)\otimes_{\Sigma_q} \big(\colim V\big)^{\otimes q}.
}$$ 
This map is an isomorphism for any $\mathbb{N}$-diagram $V$. The same argument applies to the functor taking the associated graded.
\end{proof}

\paragraph{Pushouts along free maps}
Let us now consider more general pushout diagrams of $L_\infty$-algebroids of the form
\begin{equation}\label{la:diag:pushoutalongcofibration}\vcenter{\xymatrix@R=1.8pc@C=1.9pc{
F(W\rt T_A)\ar[r]^{F(i)}\ar[d] & F(V\rt T_A)\ar[d]\\
\mf{g}\ar[r] & \mf{g}\coprod_{F(W)} F(V).
}}\end{equation}
Here $i\colon V\rt W$ is any monomorphism of dg-$A$-modules over $T_A$; the maps $V\rt T_A$ and $W\rt T_A$ need not be zero. 

We can realise Diagram \eqref{la:diag:pushoutalongcofibration} as the colimit of a pushout diagram of \emph{weakly filtered} $L_\infty$-algebroids. More precisely, let us endow the objects appearing in the above square with the following filtrations:
\begin{itemize}
\item $\mf{g}$ and $W$ have weight $\leq 0$, i.e.\ we simply take the constant $\mathbb{N}$-diagrams on $\mf{g}$ and $W$.
\item let $\widetilde{V}$ be the filtered dg-$A$-module
$$\xymatrix{
W\ar[r]^f & V\ar[r]^= & V\ar[r] & V\ar[r] & \cdots .
}$$
together with the obvious map to $T_A$ (which has weight $\leq 0$).
\end{itemize}
Diagram \eqref{la:diag:pushoutalongcofibration} is the colimit over $\mathbb{N}$ of the pushout square of weakly filtered $L_\infty$-algebroids
\begin{equation}\label{la:diag:pushoutalongfilteredcofibration}\vcenter{\xymatrix@R=1.8pc@C=1.9pc{
F(W)\ar[r]^{F(f)}\ar[d] & F(\widetilde{V})\ar[d]\\
\mf{g}\ar[r] & \mf{g}\coprod_{F(W)} F(\widetilde{V})=:\mf{h}.
}}\end{equation}
Indeed, the colimit of $\widetilde{V}$ is simply $V$ and taking colimits over $\mathbb{N}$ commutes with all colimits and free functors. On the other hand, the associated graded of \eqref{la:diag:pushoutalongfilteredcofibration} is given by
\begin{equation}\label{la:eq:gradedofalgebroidpushout}
\mm{gr}\Big(\mf{g}\coprod_{F(W)} F(\widetilde{V})\Big) \cong \mf{g}\coprod_{F(W)}F\big(\mm{gr}(\widetilde{V})\big) \cong \mf{g} \amalg F\big(V/W\big).
\end{equation}
The last isomorphism uses that the associated graded of $\widetilde{V}$ is $W\oplus V/W$, with $W$ of weight $0$ and $V/W$ of weight $1$. In particular, the map $V/W\rt T_A$ is the \emph{zero map}, since $T_A$ has weight $0$. 
\begin{proposition}\label{la:prop:reductionto0map}
Let $i\colon W\rt V$ be a monomorphism of dg-$A$-modules over $T_A$ whose cokernel is graded-free. Assume that $\mf{g}$ is an $L_\infty$-algebroid over $A$ for which the following holds: 
\begin{itemize}
\item[($\star$)] without differentials, $\mf{g}$ is a retract of the free $L_\infty$-algebroid generated by a map of graded vector spaces $M\rt T_A$.
\end{itemize}
Then the weakly filtered $L_\infty$-algebroid $\mf{h}:=\mf{g}\coprod_{F(W)} F(\widetilde{V})$ from \eqref{la:diag:pushoutalongfilteredcofibration} has the following two properties:
\begin{enumerate}
\item the weakly filtered symmetric sequence $\Env_{\mf{h}}$ is \emph{filtered}: in other words, for each $p$, there is a sequence of \emph{injections} 
$$
\Env_{\mf{h}}(p)^{(0)}\rt \Env_{\mf{h}}(p)^{(1)}\rt \Env_{\mf{h}}(p)^{(2)}\rt \cdots.
$$
\item the filtration on $\Env_{\mf{h}}$ has associated graded
$$
\mm{gr}\big(\Env_{\mf{h}}(p)\big) \cong \Env_{\mf{g}}\big(p+(-)\big)\circ_A (W/V)
$$
where $W/V$ has weight $1$ and $\Env_{\mf{g}}(p+(-))$ is as in Construction \ref{cons:symdec}.
\end{enumerate}
\end{proposition}
\begin{proof}
Let us start by verifying part (2): recall from Lemma \ref{la:lem:propertiesofenveloping}(4) that $\Env$ commutes with taking the associated graded. We then find that
$$
\mm{gr}(\Env_{\mf{h}}(p)) \cong \Env_{\mf{g}\coprod F(V/W)}(p)\cong \Env_{\mf{g}}(p+(-))\circ_A (W/V)
$$
using the isomorphism \eqref{la:eq:gradedofalgebroidpushout} and Lemma \ref{la:lem:propertiesofenveloping}(3).

To verify assertion (1), we can forget about all differentials. Without differentials, we can split $\widetilde{V}$ as a direct sum 
$$\xymatrix{
\widetilde{V}\cong W\oplus \big(A\otimes N\big) \ar[r] & T_A.
}$$
Here $N$ is a graded vector space generating the graded-free $A$-module $V/W$ (of weight $1$). The map $F(W)\rt F(\widetilde{V})$ can then be identified with the map 
$$\xymatrix{
F(W)\ar[r] & F(W)\amalg \mm{Free}(N)
}$$
into the coproduct with the free $L_\infty$-algebroid on a map of filtered $k$-modules $N\rt T_A$. By our assumption on $\mf{g}$, we can realise $\mf{h}$ as a retract of
$$
\mm{Free}(M)\coprod_{F(W)} F(\widetilde{V}) \cong \mm{Free}(M\oplus N\rt T_A).
$$
This is the free (weakly filtered) $L_\infty$-algebroid on the filtered $k$-module $M\oplus N$. In other words, it is given by $A\otimes L_\infty(M\oplus N)$. It follows that each $\Env_{\mf{h}}(p)$ is (without differentials) a retract of
\begin{align}\label{eq:envoffree}
\Env_{\mm{Free}(M\oplus N)}(p)&\cong A\otimes (L_{\infty})_{L_\infty(M\oplus N)}(p) \nonumber\\
&\cong \bigoplus_{q\geq 0} A\otimes L_\infty(p+q)\otimes_{\Sigma_q} (M\oplus N)^{\otimes q}.
\end{align}
Here the first isomorphism is from Lemma \ref{la:lem:propertiesofenveloping}(2) and the second isomorphism follows from Remark \ref{rem:envelopingoperadfreelooalgebra}. Since $M\oplus N$ is filtered, the symmetric sequence \eqref{eq:envoffree} is filtered and it follows that the retract $\Env_{\mf{h}}$ is filtered as well.
\end{proof}

\paragraph{Proof of Theorem \ref{la:thm:modelstructureondglralgebras}} 
We have to prove that the free-forgetful adjunction
$$\xymatrix{
F\colon \dgMod_A/T_A\ar@<1ex>[r] & \dgLooAlgd_A\ar@<1ex>[l]\colon U
}$$
satisfies the conditions of Lemma \ref{app:lem:transferlemma}, guaranteeing the existence of a transferred semi-model structure. Since the forgetful functor $U$ preserves filtered colimits, it suffices to show that for any cofibrant $L_\infty$-algebroid $\mf{g}$ and any generating trivial cofibration $0\rt A[n, n+1]$ in $\dgMod_A/T_A$, the map $\mf{g}\rt \mf{g}\amalg F(A[n, n+1])$ is a trivial cofibration of chain complexes:
\begin{lemma}\label{lem:hoinvarianceofenveloping}
If $\mf{g}$ is a cofibrant $L_\infty$-algebroid, then the map 
$$\xymatrix{
\mf{g}\ar[r] & \mf{g}\amalg F\big(A[n, n+1]\rt T_A\big)
}$$ 
is a trivial cofibration of chain complexes. In fact, the map of symmetric sequences $\Env_\mf{g}\rt \Env_{\mf{g}\coprod F(A[n, n+1])}$ is a levelwise trivial cofibration as well.
\end{lemma}
\begin{proof}
The generating cofibrations of $\dgLooAlgd_A$ are given by the maps 
$$\xymatrix{
F\big(A[n]\rt T_A\big)\ar[r] & F\big(A[n, n+1]\rt T_A\big)
}$$
Forgetting differentials, the pushout of an $L_\infty$-algebroid along such a map simply adds a single generator in degree $n+1$. It follows that any cofibrant $L_\infty$-algebroid $\mf{g}$ satisfies condition ($\star$) from Proposition \ref{la:prop:reductionto0map}, so that $\Env_{\mf{g}\coprod F(A[n, n+1])}(p)$ admits a filtration with associated graded
$$
\Env_\mf{g}(p+(-))\circ_A A[n, n+1].
$$
Since $A[n, n+1]$ is acyclic, it follows that the map $\Env_\mf{g}(p)\rt \Env_{\mf{g}\coprod F(A[n, n+1])}(p)$ is a trivial cofibration of chain complexes.
\end{proof}

\paragraph{Proof of Theorem \ref{la:thm:monadicity}}
We have to prove that the forgetful functor
$$\xymatrix{
U\colon \dgLooAlgd_A\ar[r] & \dgMod_A/T_A
}$$
preserves cofibrant objects and sifted homotopy colimits. Our proof will follow the lines of e.g. \cite[Lemma 4.5.4.12]{lur16}:
\begin{definition}\label{la:def:gooddiagram}
Let $\cat{M}$ be a (semi-) model category and let $X\colon \cat{J}\rt \cat{M}$ be a diagram. We will say that $X$ is \emph{good} if it satisfies the following conditions:
\begin{itemize}\setlength{\itemsep}{0pt}\setlength{\parskip}{0pt}\setlength{\parsep}{0pt}
 \item each object $X(j)$ is cofibrant.
 \item the colimit $\colim X$ is cofibrant.
 \item the map $\hocolim X\rt \colim X$ is a weak equivalence.
\end{itemize}
More generally, we say that a map of $\cat{J}$-diagrams $X\rt Y$ in $\cat{M}$ is good if
\begin{enumerate}[label=\Roman*, leftmargin=*]
\setlength{\itemsep}{0pt}
\setlength{\parskip}{0pt}
\setlength{\parsep}{0pt}
 \item[(i)] $X$ and $Y$ are both good.
 \item[(ii)] each $X(j)\rt Y(j)$ is a cofibration in $\cat{M}$.
 \item[(iii)] $\colim X\rt \colim Y$ is a cofibration in $\cat{M}$.
\end{enumerate}
\end{definition}
Let $\dgMod_{A, A^{\otimes p}}$ denote the category of dg-$A$-$A^{\otimes p}$-bimodules. We endow this category with the model structure in which a map is a weak equivalence (fibration) if and only if the underlying map of left dg-$A$-modules is a weak equivalence (fibration). With these definitions, Theorem \ref{la:thm:monadicity} follows from the following assertion by taking $p=0$:
\begin{theorem}\label{la:thm:monadicityrephrased}
Let $\cat{J}$ be a homotopy sifted category and let $\mf{g}\colon \cat{J}\rt \dgLooAlgd_A$ be a projectively cofibrant diagram (Definition \ref{def:projectivemodelstructure}). Then the diagram 
$$\xymatrix{
\Env_\mf{g}(p)\colon \cat{J}\ar[r] & \dgMod_{A, A^{\otimes p}}
}$$
is good for each $p\geq 0$.
\end{theorem}
We will prove Theorem \ref{la:thm:monadicityrephrased} by `induction on cells', using the following simple stability properties of good maps:
\begin{lemma}\label{la:lem:stabilitypropertiesofgoodmaps}
Let $\cat{J}$ be a homotopy sifted category. We have the following properties of good maps between $\cat{J}$-indexed diagrams in $\cat{M}$:
\begin{enumerate}
 \item every projectively cofibrant diagram $X\colon \cat{J}\rt \cat{M}$ is good and every projective cofibration between projectively cofibrant diagrams in $\cat{M}$ is good.
 \item good morphisms are closed under transfinite composition and retracts.
 \item given a pushout diagram
 $$\xymatrix{
 X\ar[r]^f\ar[d] & Y\ar[d]\\
 X'\ar[r]_{f'} & Y'
 }$$
 in which $f$ is good and $X'$ is good, we have that $f'$ is good.
\item If $F\colon \cat{M}\rt \cat{N}$ is a left Quillen functor and $X\rt Y$ is a good map of diagrams in $\cat{M}$, then $F(X)\rt F(Y)$ is good. 
\item If $F\colon \cat{M}_1\times\cat{M}_2\rt \cat{N}$ is a left Quillen bifunctor, $X_1\rt Y_1$ is a good map of $\cat{J}_1$-diagrams in $\cat{M}_1$ and $X_2\rt Y_2$ is a good map of $\cat{J}_2$-diagrams in $\cat{M}_2$, then $F(X_1, X_2)\rt F(Y_1, Y_2)$ is a good map of $\cat{J}_1\times\cat{J}_2$-diagrams.
 \item If $X\colon \cat{J}^{\times n}\rt \cat{M}$ is good, then the restriction along the diagonal $\Del^*X\colon \cat{J}\rt \cat{M}$ is good.
\end{enumerate}
When $\cat{M}=\dgMod_{A, A^{\otimes p}}$ is the model category of dg-$A$-$A^{\otimes p}$-bimodules, we furthermore have:
\begin{enumerate}
 \item[(7)] If $X$ is a good diagram with a $\Sigma_n$-action, then $X/\Sigma_n$ is good.

 \item[(8)] Let $f\colon X\rt Y$ be a natural monomorphism with a good domain and a good cokernel, such that the map $\colim X\rt \colim Y$ is a monomorphism. Then $f$ is good.
\end{enumerate}
\end{lemma}
\begin{proof} The first four properties are easily verified. For (5), it is clear that the map $F(X_1, X_2)\rt F(X_1, Y_2)\rt F(Y_1, Y_2)$ is a pointwise cofibration. The map on colimits
$$\xymatrix{
\colim_{\cat{J}_1\times\cat{J}_2} F(X_1, X_2) \ar[r] & \colim_{\cat{J}_1\times\cat{J}_2} F(Y_1\times Y_2)
}$$
is a cofibration between cofibrant objects. Indeed, it is the image under $F$ of the maps
$$\xymatrix{
\colim X_1\ar[r] & \colim Y_1 & \colim X_2\ar[r] & \colim Y_2
}$$
which are both cofibrations between cofibrant objects. To see that $F(X_1, X_2)$ is good, note that there are weak equivalences between cofibrant objects
\begin{align*}
\hocolim_{\cat{J}_1\times\cat{J}_2} F(X_1, X_2) &\simeq \hocolim_{\cat{J}_1} F\big(X_1, \hocolim_{\cat{J}_2}(X_2)\big)\\
&\simeq F\big(\hocolim_{\cat{J}_1}(X_1), \hocolim_{\cat{J}_2}(X_2)\big)\\
&\simeq F\big(\colim_{\cat{J}_1}(X_1), \colim_{\cat{J}_2}(X_2)\big)\cong \colim_{\cat{J}_1\times\cat{J}_2} F(X_1, X_2).
\end{align*}
For (6), clearly $\Del^*X$ is pointwise cofibrant. Since $\Del$ is homotopy cofinal, the (homotopy) colimit of $\Del^*X$ agrees with the (homotopy) colimit of $X$, which proves conditions (ii) and (iii). Assertion (7) follows from the fact that we are working in characteristic zero, so that the colimit functor $\colim\colon \cat{M}^{\Sigma_n}\rt \cat{M}$ is left Quillen for the \emph{injective} model structure. 

Finally, for (8) we use that a map of dg-$A$-$A^{\otimes p}$-bimodules is a cofibration if and only if it is a monomorphism with cofibrant cokernel. This implies that the maps $X(i)\rt Y(i)$ and $\colim X\rt \colim Y$ are cofibrations, so that all $Y(i)$ and $\colim Y$ are cofibrant. Furthermore, we obtain a commuting diagram
$$\xymatrix@R=1.7pc@C=1.9pc{
\hocolim X\ar[d]_\sim \ar[r] & \hocolim Y\ar[r] \ar[d] & \hocolim Y/X\ar[d]^\sim\\
\colim X\ar[r] & \colim Y \ar[r] & \colim Y/X
}$$
Both horizontal sequences are cofibre sequences of dg-$A$-modules, so the map $\hocolim Y\rt\colim Y$ is a weak equivalence.
\end{proof}
\begin{lemma}\label{lem:attachingcelltodiagram}
Let $\cat{J}$ be a homotopy sifted category and consider a pushout square
$$\xymatrix@R=1.6pc@C=1.5pc{
F(W)\ar[rr]^-{\mm{Free}(i)}\ar[d] & & F(V)\ar[d]\\
\mf{g}\ar[rr] & & \mf{h}
}$$
where $i$ is a projective cofibration of $\cat{J}$-diagrams of dg-$A$-modules and $\mf{g}$ is a projectively cofibrant diagram of $L_\infty$-algebroids. If each $\mm{Env}_\mf{g}(p)$ is good, then each $\mm{Env}_\mf{g}(p)\rt \mm{Env}_{\mf{h}}(p)$ is good.
\end{lemma}
\begin{proof}
Note that each $L_\infty$-algebroid $\mf{g}(j)$ is cofibrant, so that it satisfies condition ($\star$) of Proposition \ref{la:prop:reductionto0map}. It follows that there is a natural filtration 
$$\xymatrix{
\Env_\mf{g}(p)=\Env_{\mf{h}}(p)^{(0)}\ar[r] & \Env_{\mf{h}}(p)^{(1)}\ar[r] & \cdots\ar[r] & \mm{Env}_{\mf{h}}(p)
}$$
on the $\cat{J}$-diagram $\Env_\mf{h}(p)$. There is a similar filtration on the colimit, because $\colim(\mf{g})$ is a cofibrant $L_\infty$-algebroid as well. By parts (2) and (8) of Lemma \ref{la:lem:stabilitypropertiesofgoodmaps}, it then suffices to verify that the associated graded
\begin{equation}\label{diag:diagofgraded}
\bigoplus_{q\geq 0} \Env_{\mf{g}}(p+q)\otimes_{A^{\otimes q}} (V/W)^{\otimes q}\Big/\Sigma_q
\end{equation}
consists of good $\cat{J}$-diagrams of dg-$A$-$A^{\otimes p}$-bimodules. 

Let us abbreviate the $\cat{J}$-diagram of $A$-$A^{\otimes p+q}$-bimodules $\Env_{\mf{g}}(p+q)$ to just $E$. To show that the above sum consists of good $\cat{J}$-diagrams, consider the functor
$$\vcenter{\xymatrix{
T\colon \cat{J}^{\times 1+q}\ar[r] & \dgMod_{A, A^{\otimes p}}; \hspace{4pt} (i, j_1, \dots, j_q) \ar@{|->}[r] & E_i\otimes_{A^{\otimes q}}(V/W_{j_1}\otimes \dots V/W_{j_q}).
}}$$
We claim that $T$ is a good diagram of dg-$A$-$A^{\otimes p}$-bimodules. To see this, note that the functor
$$\xymatrix{
\dgMod_{A, A^{\otimes r}}\times \dgMod_{A}\ar[r] & \dgMod_{A, A^{\otimes r-1}}; \hspace{4pt} (E, V)\ar@{|->}[r] & E\otimes_{A} V
}$$
is a left Quillen bifunctor. Since $E$ is a good diagram of $A$-$A^{\otimes p+q}$-bimodules and $V/W$ is a good (projectively cofibrant) diagram of dg-$A$-modules, a repeated application of part (5) of Lemma \ref{la:lem:stabilitypropertiesofgoodmaps} shows that the functor $T$ is good.

Restricting $T$ along the diagonal and taking the quotient by the action of the symmetric group, we obtain a good $\cat{J}$-diagram of dg-$A$-modules by parts (6) and (7) of Lemma \ref{la:lem:stabilitypropertiesofgoodmaps}. This means that the associated graded \eqref{diag:diagofgraded} is a sum of good diagrams, which proves the result. 
\end{proof}
\begin{proof}[Proof (of Theorem \ref{la:thm:monadicityrephrased})]
Let $\mc{K}$ be the subcategory of $\mf{g}\colon \cat{J}\rt \dgLooAlgd_A$ for which each diagram of dg-$A$-modules $\Env_\mf{g}(p)$ is good. By part (2) of Lemma \ref{la:lem:stabilitypropertiesofgoodmaps}, $\mc{K}$ is closed under retracts and contains the colimit of a transfinite sequence $\mf{g}_\sbullet$ for which each $\mm{Env}_{\mf{g}_{\alpha}}(p)\rt \Env_{\mf{g}_\beta}(p)$ is good. It therefore suffices to show that $\mc{K}$ is closed under pushouts along generating cofibrations, which is Lemma \ref{lem:attachingcelltodiagram}.
\end{proof}
\begin{remark}\label{rem:fortame}
The above proofs apply verbatim in situations where the base monoidal model category $\dgMod_k$ of chain complexes (over a field $k$ of characteristic zero) is replaced by the monoidal model category of graded-mixed complexes or presheaves of complexes over $k$ (Variants \ref{var:gradedmixed} and \ref{var:presheaves}). In both cases, one can directly transfer the model structure along the forgetful functor to (graded mixed, presheaves of) chain complexes over $T_A$, using the filtration of Proposition \ref{la:prop:reductionto0map} (at each object in the site).

One can also obtain a semi-model structure on $L_\infty$-algebroids by transfer from the \emph{tame} model structure on $\dgMod_A$, which itself is not transferred from $\dgMod_k$ (see Variant \ref{var:tame}). Indeed, the tame model structure on $\dgMod_A$ is a monoidal model structure, which has the same trivial cofibrations as the projective model structure and whose cofibrations are generated by inclusions $V\rt V[0, 1]$ of graded-free dg-$A$-modules into their cone \cite[Lemma 3.4]{nui17c}. It follows that all cofibrant $L_\infty$-algebroids still satisfy condition ($\star$) from Proposition \ref{la:prop:reductionto0map}. 

The above proofs apply now apply verbatim to this case as well. Let us remark that in Theorem \ref{la:thm:monadicityrephrased}, one uses the model structure on dg-bimodules transferred from the tame model structure on left dg-$A$-modules, which is different from the tame model structure on $A$-$A^{\otimes p}$-bimodules.
\end{remark}

\section{Cofibrant replacement}\label{la:sec:cofibrantreplacement}
There is no straightforward way to replace a dg-Lie algebroid or $L_\infty$-algebroid by a fibrant dg-Lie algebroid; this is the main reason for the non-existence of an actual model structure on dg-Lie algebroids. The purpose of this section is to provide a reasonably concrete \emph{cofibrant} replacement for dg-Lie algebroids and $L_\infty$-algebroids, which is analogous to the cobar resolution for algebras over reduced operads. 

Let us start by briefly recalling the relation between $L_\infty$-algebras and cocommutative dg-coalgebras. All cocommutative coalgebras are assumed to be without counit and conilpotent (every element is annihilated by some $n$-fold composite of the comultiplication). For any cocommutative coalgebra $C$ and an $L_\infty$-algebra $\mf{h}$, the chain complex $\Hom(C, \mf{h})$ has the structure of an $L_\infty$-algebra, with differential given by $\dau\tau= \dau_\mf{h}\circ \tau - \tau\circ \dau_{C}$ and $n$-ary bracket given by composing the $n$-ary bracket in $\mf{h}$ with the $n$-fold comultiplication in $C$. A \emph{twisting cochain} is a Maurer-Cartan element of this $L_\infty$-algebra, i.e.\ a map $C\rt \mf{h}[1]$ which satisfies the Maurer-Cartan equation
\begin{equation}\label{la:eq:maurercartan}
\dau\tau + \sum_{j\geq 2} \frac{1}{j!} [\tau, \dots , \tau]_j =0.
\end{equation}
The infinite sum is well-defined because $C$ is conilpotent. There are natural bijections
\begin{equation}\label{diag:twistinggivesadj}
\Hom_{\dgLooAlg_k}(\mm{\Omega} C, \mf{g}) \cong \mm{Twist}(C, \mf{g})\cong \Hom_{\cat{CoAlg}_k^{\dg}}\big(C, \ol{C}_*(\mf{g})\big)
\end{equation}
between the set of twisting cochains $C\rt \mf{g}[1]$, the set of maps of $L_\infty$-algebras $\mm{\Omega} C\rt \mf{g}$ from the \emph{cobar} construction of $C$ and the set of maps of cocommutative coalgebras $C\rt \ol{C}_*(\mf{g})$ to the \emph{reduced (homological) Chevalley-Eilenberg complex} of $\mf{g}$. The latter is the cofree (conilpotent, non-counital) graded-cocommutative coalgebra
\begin{equation}\label{eq:coprod}
\ol{C}_*(\mf{g}):= \mm{Sym}^{\sgeq 1}_k \mf{g}[1] \hspace{35pt} \Delta(x_1\, .\, .\, x_n) = \sum_{i, \sigma} x_{\sigma(1)}\, .\, .\, x_{\sigma(i)} \otimes x_{\sigma(i+1)}\, .\, .\, x_{\sigma(n)}
\end{equation}
where the sum runs over all $(i, k-i)$-unshuffles, endowed with the unique differential extending the map
$$\xymatrix{
\sum_{n\sgeq 1} [-, \dots, -]_n\colon \mm{Sym}^{\sgeq 1}_k \mf{g}[1]\ar[r] & \mf{g}[1].
}$$
The natural isomorphisms \eqref{diag:twistinggivesadj} realise the cobar functor $\mm{\Omega}$ as a left adjoint to $\ol{C}_*$.
\begin{definition}[{cf.\ \cite[Section 13.2.12]{lod12}}]\label{la:def:nonlinearmaplooalgebras}
An \emph{$\infty$-morphism} of $L_\infty$-algebras $\mf{g}\leadsto \mf{h}$ (also called an \emph{$L_\infty$-morphism}) is a twisting cochain 
$$
\ol{C}_*(\mf{g})\rt \mf{h}[1],
$$
or equivalently, a map of cocommutative dg-coalgebras $\ol{C}_*(\mf{g})\rt \ol{C}_*(\mf{h})$.
\end{definition}
\begin{definition}\label{la:def:nonlinearmap}
Let $\mf{g}$ and $\mf{h}$ be $L_\infty$-algebroids over $A$. An \emph{$\infty$-morphism} $\mf{g}\leadsto \mf{h}$ of $L_\infty$-algebroids is an $\infty$-morphism of $L_\infty$-algebras $\tau\colon\mf{g}\leadsto \mf{h}$, such that
\begin{itemize}\setlength{\itemsep}{2pt}
\item[(i)] the composite map $\rho_\mf{h}(\tau)\colon \ol{C}_*(\mf{g})\rt \mf{h}[1]\rt T_A[1]$ first takes the quotient by $\mm{Sym}^{\sgeq 2}_k\mf{g}[1]\subseteq \ol{C}_*(\mf{g})$ and then applies the anchor of $\mf{g}$ to the remaining $\mf{g}[1]$.
\item[(ii)] the map of graded vector spaces $\tau\colon \mm{Sym}^{\sgeq 1}_k(\mf{g}[1])\rt \mf{h}[1]$ descends to a graded $A$-linear map $\mm{Sym}_A^{\sgeq 1}(\mf{g}[1])\rt \mf{h}[1]$.
\end{itemize}
Let $\LooAlgd_A^{(\infty)}$ denote the category of $L_\infty$-algebroids and $\infty$-morphisms between them.
\end{definition}
\begin{remark}
The category of $L_\infty$-algebras and $\infty$-morphisms between them is a full subcategory of the category of cocommutative dg-coalgebras (on the fibrant objects in the model structure from \cite{hin98}). We do not know if $\LooAlgd_A^{(\infty)}$ can be embedded into such a category of coalgebraic objects.
\end{remark}
\begin{remark}
$\infty$-morphisms between $L_\infty$-algebroids frequently arise in the setting of Remark \ref{rem:finiterank}, where $A$ is an ordinary ring and the $L_\infty$-algebroids considered are nonnegatively graded complexes of finite rank projective $A$-modules. In this case, $\infty$-morphisms $\mf{g}\leadsto \mf{h}$ can be identified with maps of cdgas $\mm{Sym}_A(\mf{h}[1])^\vee\rt \mm{Sym}_A(\mf{g}[1])^\vee$, also known as `maps of NQ-supermanifolds'. For example, they naturally appear when studying homotopy transfer of $L_\infty$-algebroid structures \cite{pym16}.
\end{remark}
Every strict morphism between $L_\infty$-algebroids defines an $\infty$-morphism, using the functoriality of $\ol{C}_*(\mf{g})$ in maps of $L_\infty$-algebras. We therefore obtain a functor 
$$
\iota\colon \dgLooAlgd_A\rt \LooAlgd_A^{(\infty)}
$$
which is the identity on objects.
\begin{lemma}
The functor $\iota$ admits a left adjoint $Q\colon \LooAlgd_A^{(\infty)}\rt \dgLooAlgd_A$.
\end{lemma}
\begin{proof}
For each $\mf{g}\in\LooAlgd_A^{(\infty)}$, it suffices to prove that the functor
$$\xymatrix{
\dgLooAlgd_A\ar[r] & \Set; \hspace{4pt} \mf{h}\ar@{|->}[r]&  \big\{\infty\text{-morphisms } \mf{g}\leadsto \mf{h}\big\}
}$$
can be corepresented by an $L_\infty$-algebroid $Q(\mf{g})$. This functor sends limits of $L_\infty$-algebroids (and strict maps between them) to limits of sets, and similarly for $\kappa$-filtered colimits where $\kappa$ exceeds the cardinality of $\mf{g}$. The existence of a corepresenting object $Q(\mf{g})$ then follows from the following version of the adjoint functor theorem, sometimes called `the representability theorem': any accessible, limit-preserving functor $F\colon\cat{C}\rt \Set$ out of a locally presentable category is corepresentable (see e.g.\ \cite[Proposition 5.5.2.7]{lur09}, whose proof also applies to locally presentable categories instead of $\infty$-categories).
\end{proof}
Our goal will be to show that $Q(\mf{g})$ often provides a cofibrant replacement for the $L_\infty$-algebroid $\mf{g}$, see Proposition \ref{la:prop:cobariscofibrant} and Proposition \ref{la:prop:propertiesofcobar}(c). To this end, let us first give a somewhat more explicit description of the $L_\infty$-algebroid $Q(\mf{g})$.
\begin{construction}
Let $\mf{g}$ be an $L_\infty$-algebroid and consider the maps of graded vector spaces
$$\xymatrix{
\mm{Sym}^{\sgeq 1}_k\big(\mf{g}[1]\big)[-1]\ar[r] & \mm{Sym}^{\sgeq 1}_A\big(\mf{g}[1]\big)[-1]\ar[r]^-{\mm{proj}} & \mf{g} \ar[r]^\rho & T_A.
}$$
The last two maps are graded $A$-linear. Working without differentials, we can take free $L_\infty$-algebroids and obtain a surjective map
\begin{equation}\label{diag:quot}\xymatrix{
q\colon A\otimes \mm{\Omega}\ol{C}_*(\mf{g})\cong\mm{Free}\Big(\mm{Sym}^{\sgeq 1}_k\big(\mf{g}[1]\big)[-1]\Big)\ar@{->>}[r] & F\Big(\mm{Sym}^{\sgeq 1}_A(\mf{g}[1])[-1]\Big).
}\end{equation}
As indicated, the domain of this map is isomorphic to the action $L_\infty$-algebroid associated to the operadic bar-cobar resolution $\mm{\Omega}\ol{C}_*(\mf{g})\rt \mf{g}\rt T_A$ of the $k$-linear $L_\infty$-algebra underlying $\mf{g}$. In particular, it comes equipped with a differential.

Unwinding the definition of the differential on the bar-cobar construction (see e.g.\ \cite[Section 11.2.2, 11.2.8]{lod12}), one sees that on generators $x_1\dots x_n\in \mm{Sym}^{\sgeq 1}_k\big(\mf{g}[1]\big)[-1]$, the differential takes the form
\begin{equation}\label{eq:totaldiff}
\dau\big(x_1\dots x_n\big) = \dau_\mm{lin}(x_1\dots x_n) + \kappa(x_1\dots x_n).
\end{equation}
The first term $\dau_\mm{lin}(x_1\dots x_n) = \sum x_1\dots \dau_\mf{g}(x_i)\dots x_n$ is the differential induced from the differential on $\mf{g}$. The second term is given (modulo Koszul signs) by
\begin{equation}\label{eq:connectinghom}
\kappa\big(x_1\dots x_k\big) = -\sum_{i\geq 2}\sum_{\sigma} [x_{\sigma(1)}, \dots, x_{\sigma(i)}]_{\mf{g}} \dots x_{\sigma(k)}+ \sum_{j\geq 2}\frac{1}{j!} [\Del^{(j)}(x_1\dots x_k)].
\end{equation}
Here $\sigma$ runs over the $(i, k-i)$-unshuffles. The bracket $[-, \dots, -]_\mf{g}$ denotes the bracket of $\mf{g}$, while $[-,\dots, -]$ denotes the (formal) bracket in $A\otimes \mm{\Omega}\ol{C}_*(\mf{g})$. Finally, $\Del^{(j)}$ denotes the $j$-fold comultiplication on $\mm{Sym}^{\sgeq 1}_k(\mf{g}[1])$.

To see equation \eqref{eq:connectinghom}, recall that the differential on $\mm{\Omega}\ol{C}_*(\mf{g})$ is the sum of the differential induced from $\ol{C}_*(\mf{g})[-1]$, and a second differential built from the comultiplication of $\ol{C}_*(\mf{g})$ and the brackets in $\mm{\Omega}\ol{C}_*(\mf{g})$. This second differential is exactly the second term of $\kappa$. Similarly, the differential on $\ol{C}_*(\mf{g})[-1]$ is build out of the differential on $\mf{g}$ and the comultiplication on $\ol{C}_*(\mf{g})$. The first contributes $\dau_\mm{lin}$ to the differential and the second contributes the first term of $\kappa$ \eqref{eq:connectinghom} (the minus sign arises because we pass from $\ol{C}_*(\mf{g})$ to its desuspension).
\end{construction}
\begin{lemma}\label{lem:Qdescribed}
The differential \eqref{eq:totaldiff} passes to the quotient $F\big(\mm{Sym}^{\sgeq 1}_A\big(\mf{g}[1]\big)[-1]\big)$. Consequently, we have that
$$
Q(\mf{g}) \cong \bigg(F\Big(\mm{Sym}^{\sgeq 1}_A\big(\mf{g}[1]\big)[-1]\Big), \dau= \dau_{\mm{lin}}+\kappa\bigg).
$$
\end{lemma}
\begin{proof}
For the second part of the statement, it suffices to check the universal property. To this end, note that maps of $L_\infty$-algebroids without differentials
$$\xymatrix{
f_\tau\colon F\Big(\mm{Sym}^{\sgeq 1}_A(\mf{g}[1])[-1]\Big)\ar[r] & \mf{h}
}$$ 
correspond to bijectively to graded $A$-linear maps $\tau\colon \mm{Sym}_A^{\sgeq 1}\mf{g}[1]\rt \mf{h}[1]$ such that the composite map $\mm{Sym}_A^{\sgeq 1}\mf{g}[1]\rt T_A[1]$ first projects to $\mf{g}[1]$ and then applies the anchor. The map $f_\tau$ preserves differentials iff its restriction to $A\otimes \mm{\Omega}\ol{C}_*(\mf{g})$ preserves differentials, i.e.\ iff the composite $\ol{C}_*(\mf{g})\rt \mm{Sym}^{\sgeq 1}_A\mf{g}[1]\rt \mf{h}[1]$ is a twisting cochain. It follows that such maps $f_\tau$ correspond precisely to $\infty$-morphisms $\mf{g}\leadsto \mf{h}$.

For the first part, comparing universal properties shows that without the differential, $F\big(\mm{Sym}^{\sgeq 1}_A\big(\mf{g}[1]\big)[-1]\big)$ is given by the quotient
$$
\mm{Free}\Big(\mm{Sym}^{\sgeq 1}_k\big(\mf{g}[1]\big)[-1]\Big)\bigg/\Big(1\otimes x_1\dots (ax_i)\dots x_k - a\otimes x_1\dots x_k\Big)
$$
by the graded ideal generated by $a\otimes x_1\dots x_k - 1\otimes x_1\dots (ax_i)\dots x_k$ (see Example \ref{ex:ideal}). It suffices to check that this ideal is also closed under the differential \eqref{eq:totaldiff}. Unravelling the definitions, one sees that it is enough to check that the following two equations hold in $F\big(\mm{Sym}^{\sgeq 1}_A\big(\mf{g}[1]\big)[-1]\big)$, ignoring Koszul signs:
\begin{align}\label{eq:kappaisalinear}
\dau_{\mm{lin}}(x_1\dots (ax_i)\dots x_k) - a\cdot \dau_\mm{lin}(x_1\dots x_n) &= \dau_A(a)\cdot x_1\dots x_n\nonumber \\
\kappa\big(x_1\dots (ax_i)\dots x_k\big)-a\cdot \kappa\big(x_1 \dots x_k\big)&=0.
\end{align}
The first equation is easily verified. The second equation asserts that $\kappa$ depends $A$-multilinearly on the $x_i$. Since it is symmetric in the $x_i$ by assumption, it suffices to take $i=1$.

To see that $\kappa$ depends $A$-multilinear on the $x_i$, recall that all brackets of arity $\geq 3$ are $A$-multilinear. Consequently, most terms in \eqref{eq:connectinghom} are already $A$-multilinear, so that the left hand side of \eqref{eq:kappaisalinear} reduces to
\begin{align}\label{eq:connectinghomlinearity}
\kappa\big((ax_1)\dots x_k\big)-a\cdot \kappa\big(x_1 \dots x_k\big) &= \sum_{j>1} [ax_1, x_j] \dots x_k - a[x_1, x_j]\dots x_k\nonumber\\
&- \frac{1}{2} \big[\Del\big((ax_1)\dots x_k\big)\big]+ \frac{1}{2} a \big[\Del(x_1\dots x_k)\big]
\end{align}
Next, note that taking brackets with an element from $\big(\mm{Sym}_A^{\sgeq 2}\mf{g}[1]\big)[-1]$ is $A$-linear, since the latter space is contained in the kernel of the anchor map of $F\big(\mm{Sym}_A^{\sgeq 2}\mf{g}[1]\big)[-1]\big)$. It follows that the second line of \eqref{eq:connectinghomlinearity} reduces to
$$
 - \frac{1}{2} \big[\Del\big((ax_1)\dots x_k\big)\big] + \frac{1}{2} a [\Del(x_1\dots x_k)] = \sum_{j>1} [ax_1\, .\, .\, \hat{x}_j\, .\, .\, x_n, x_j] - a [x_1\, .\, .\, \hat{x}_j\,  .\, .\, x_n, x_j].
$$
The right hand side consists only of the contributions to coproduct \eqref{eq:coprod} from the $(1, k-1)$-unshuffles and $(k-1, 1)$-unshuffles. Applying the Leibniz rule, one then sees that the two lines in \eqref{eq:connectinghomlinearity} cancel out, so that $\kappa$ is indeed $A$-multilinear.
\end{proof}
\begin{definition}\label{la:def:acofibrantliealgd}
An $L_\infty$-algebroid $\mf{g}$ is \emph{$A$-cofibrant} if its underlying dg-$A$-module is cofibrant.
\end{definition}
\begin{proposition}\label{la:prop:cobariscofibrant}
If $\mf{g}$ is an $A$-cofibrant $L_\infty$-algebroid, then $Q(\mf{g})$ is cofibrant.
\end{proposition}
\begin{proof}
For any $L_\infty$-algebroid $\mf{g}$, let us define 
$$
Q^{(n)}(\mf{g})=F\Big(\mm{Sym}_A^{1\leq \bullet\leq n}(\mf{g}[1])[-1]\Big) \subseteq Q(\mf{g}).
$$
By formulas \eqref{eq:totaldiff} and \eqref{eq:connectinghom}, the differential on $Q(\mf{g})$ preserves $Q^{(n)}(\mf{g})$, so that $Q^{(n)}$ is indeed an $L_\infty$-algebroid. We therefore obtain a sequence of inclusions
\begin{equation}\label{la:diag:unionofcobars}\vcenter{\xymatrix{
Q^{(1)}(\mf{g})\ar[r] & Q^{(2)}(\mf{g})\ar[r] & \dots \ar[r] & Q(\mf{g})
}}\end{equation}
whose colimit is $Q(\mf{g})$. Note that $\kappa$ vanishes on generators from $\mf{g}\subseteq \mm{Sym}_A^{\sgeq 1}(\mf{g}[1])[-1]$. Consequently, $Q^{(1)}(\mf{g})$ is simply the free $L_\infty$-algebroid generated by the $A$-linear map $\rho\colon \mf{g}\rt T_A$; this is certainly cofibrant if $\mf{g}$ is $A$-cofibrant.

To prove the result, it suffices to verify that each map $Q^{(n)}(\mf{g})\rt Q^{(n+1)}(\mf{g})$ is a cofibration. To this end, observe that there is a pushout square of $L_\infty$-algebroids of the form
\begin{equation}\label{la:diag:attachingcellincobar}\vcenter{\xymatrix@R=1.8pc{
F\Big(\big(\mm{Sym}_A^{n+1}\mf{g}[1]\big)[-2]\Big)\ar[d]\ar[r]^-{\kappa} & Q^{(n)}(\mf{g})\ar[d]\\
F\Big(\big(\mm{Sym}_A^{n+1}\mf{g}[1]\big)[-2, -1]\Big)\ar[r] & Q^{(n+1)}(\mf{g}).
}}\end{equation}
Here the left two $L_\infty$-algebroids are freely generated by the twofold desuspension of the dg-$A$-module $\mm{Sym}_A^{n+1}\mf{g}[1]$ and its cone, both equipped with the zero anchor map. Indeed, unravelling the universal property of the pushout, one obtains the following description of the pushout of \eqref{la:diag:attachingcellincobar}. One freely adds generators from $\mm{Sym}_A^{n+1}(\mf{g}[1]\big)[-1]$ to $Q^{(n)}(\mf{g})$. The differential of a new generator $x_1\dots x_{n+1}$ is the sum
$$
\sum_{i} x_1\dots \dau_\mf{g}(x_i)\dots x_{n+1} + \kappa(x_1\dots x_{n+1}).
$$
The first term the `internal' differential in $\mm{Sym}_A^{n+1}(\mf{g}[1]\big)[-1]$ and the second term is the attaching map $\kappa$. More precisely, the above sum is obtained by taking the differential of $x_1\dots x_{n+1}$ in the cone $\big(\mm{Sym}_A^{n+1}\mf{g}[1]\big)[-2, -1]$ and applying the attaching map $\kappa$ to the part that is contained in $\big(\mm{Sym}_A^{n+1}\mf{g}[1]\big)[-2]$. The result is exactly $Q^{(n+1)}(\mf{g})$, together with its differential \eqref{eq:totaldiff}.

When $\mf{g}$ is $A$-cofibrant, $\mm{Sym}_A^{n+1}\mf{g}[1]$ is cofibrant as well, so that the left vertical map in \eqref{la:diag:attachingcellincobar} is a cofibration. It follows that each $Q^{(n)}(\mf{g})\rt Q^{(n+1)}(\mf{g})$ is a cofibration, so that the colimit $Q(\mf{g})$ is cofibrant. 
\end{proof}
\begin{remark}\label{rem:naturalityoffiltration}
Note that the filtration \eqref{la:diag:unionofcobars} on $Q(\mf{g})$ depends functorially on $\mf{g}$, with respect to the $\infty$-morphisms. Indeed, an $\infty$-morphism of $L_\infty$-algebroids $\mf{g}\leadsto \mf{h}$ induces a map of coalgebras (without differential) $\mm{Sym}_A^{\sgeq 1}\mf{g}[1]\rt \mm{Sym}_A^{\sgeq 1}\mf{h}[1]$, which induces the maps on the various stages of the filtration. 

Note that $\kappa$ decreases the filtration degree and hence vanishes on the associated graded of the filtration. Using this, it follows that the associated graded of the filtration \eqref{la:diag:unionofcobars} is given by
\begin{equation}\label{eq:gradedoffiltration}
\mm{gr}\Big(Q^{(\sbullet)}(\mf{g})\Big) \cong F\Big(\mm{Sym}_A^{\geq 1}(\mf{g}[1])[-1]\Big).
\end{equation}
Here the right hand side is endowed with the differential $\dau_\mm{lin}$ induced from the differential on $\mf{g}$, i.e.\ it is the free $L_\infty$-algebroid generated $\mm{Sym}_A^{\geq 1}(\mf{g}[1])[-1]$. Furthermore, its (weight) grading is inherited from the grading on $\mm{Sym}_A^{\geq 1}(\mf{g}[1])$ by polynomial degree, using the `free graded $L_\infty$-algebroid' functor from Section \ref{la:sec:technicalities}.
\end{remark}
\begin{proposition}\label{la:prop:propertiesofcobar}
The functor $Q\colon \LooAlgd_A^{(\infty)}\rt \dgLooAlgd_A$ enjoys the following properties:
\begin{itemize}
 \item[(a)] let $\mf{g}\leadsto \mf{h}$ be an $\infty$-morphism between $A$-cofibrant $L_\infty$-algebroids. If the linear part $\mf{g}\rt \mf{h}$ is a weak equivalence, then $Q(\mf{g})\rt Q(\mf{h})$ is a weak equivalence.
 \item[(b)] the composite functor from the category of $A$-cofibrant $L_\infty$-algebroids
 $$\smash{\xymatrix{
 Q\circ \iota\colon \LooAlgd_A^{\dg, \mm{A-cof}}\ar[r] & \LooAlgd_A^{\dg, \mm{cof}}
 }}$$
 is a relative functor that preserves $\Del^{\mm{op}}$-indexed homotopy colimits.
 \item[(c)] the counit map $Q(\mf{g})\rt \mf{g}$ is a weak equivalence whenever $\mf{g}$ is $A$-cofibrant.
\end{itemize}
\end{proposition}
\begin{proof}
For assertion (a), using the filtration \eqref{la:diag:unionofcobars} and Remark \ref{rem:naturalityoffiltration}, it suffices to check that the map on the associated graded is a weak equivalence. Note that the higher order components of the $\infty$-morphism $\mf{g}\leadsto \mf{h}$ decrease the filtration degree and hence induce the zero map on the associated graded. In light of the isomorphism \eqref{eq:gradedoffiltration}, it then suffices to verify that the linear part $\mf{g}\rt \mf{h}$ induces graded quasi-isomorphisms 
$$\xymatrix{
F\Big(\mm{Sym}_A^{\geq 1}(\mf{g}[1])[-1]\Big)\ar[r] & F\Big(\mm{Sym}_A^{\geq 1}(\mf{h}[1])[-1]\Big).
}$$
This follows because each $\mm{Sym}_A^{p}(\mf{g}[1])\rt \mm{Sym}_A^p(\mf{h}[1])$ is a quasi-isomorphism between cofibrant dg-$A$-modules when $\mf{g}\rt \mf{h}$ is a quasi-isomorphism between cofibrant dg-$A$-modules.

For part (b), note that the composite functor $Q\circ \iota$ preserves weak equivalences between $A$-cofibrant $L_\infty$-algebroids by part (a). Suppose that $\mf{g}_\sbullet\colon \Del^\mm{op}\rt \dgLooAlgd_A$ is a projectively cofibrant diagram. In particular, each $\mf{g}_i$ is a cofibrant $L_\infty$-algebroid and hence $A$-cofibrant by Theorem \ref{la:thm:monadicity}(a). We have to show that the natural map
$$\xymatrix{
\hocolim\big(Q(\mf{g}_\sbullet)\big)\ar[r] & \colim\big(Q(\mf{g}_\sbullet)\big)\ar[r] & Q(\colim\mf{g}_\sbullet)
}$$
is a quasi-isomorphism. The description of $Q$ from Lemma \ref{lem:Qdescribed} shows that the second map is an isomorphism. To see that the left map is a quasi-isomorphism, note that we can work at the level of the underlying chain complexes by Theorem \ref{la:thm:monadicity}(b). 
The filtration \eqref{la:diag:unionofcobars} on $Q(\mf{g}_\sbullet)$ induces a filtration on the (homotopy) colimit, and it suffices to check that the map on the associated graded is a quasi-isomorphism. By a similar argument as in (a), and using that taking the associated graded commutes with (homotopy) colimits, it suffices to show that
$$\xymatrix{
\hocolim F\Big(\mm{Sym}_A^{\geq 1}(\mf{g}_\sbullet[1])[-1]\Big)\ar[r] & \colim F\Big(\mm{Sym}_A^{\geq 1}(\mf{g}_\sbullet[1])[-1]\Big)
}$$
is a graded quasi-isomorphism. The left Quillen functor $F$ and taking symmetric powers both commute with taking sifted (homotopy) colimits (by Lemma \ref{la:lem:stabilitypropertiesofgoodmaps}(5) and (7)). The result then follows from the fact that $\hocolim\mf{g}\simeq \colim\mf{g}$, since we assumed that $\mf{g}$ was projectively cofibrant.

For part (c), note that by parts (a) and (b), together with Corollary \ref{la:cor:resolutionbyliealgebras}, it suffices to prove this when $\mf{g}=A\otimes_k \mf{h}$ is just the $A$-linear extension of an ordinary $L_\infty$-algebra $\mf{h}$ over $k$. In that case, the map $Q(\mf{g})\rt \mf{g}$ is just the $A$-linear extension of the usual map $\mm{\Omega}(\ol{C}_*(\mf{h}))\rt \mf{h}$ of $L_\infty$-algebras from the operadic cobar construction of $\mf{h}$. This map is a weak equivalence (see e.g.\ \cite{lod12} for a textbook account).
\end{proof}
As usual, the derived mapping space between two $L_\infty$-algebroids can be described by the simplicial set of maps from a cofibrant replacement of the domain to a fibrant simplicial resolution of the codomain. Such a simplicial resolution of fibrant $L_\infty$-algebroids has been described in \cite{vez13}:
\begin{construction}[\cite{vez13}]\label{la:cons:simplicialcotensor}
Let $\mf{g}$ be an $L_\infty$-algebroid over $A$ and let $B$ be any (possibly unbounded) commutative dg-algebra over $k$. Then $\mf{g}\otimes_k B$ has the structure of an $A$-module and an $L_\infty$-algebra and the anchor map extends to a $B$-linear map $\mf{g}\otimes_k B\rt T_A\otimes_k B$. Let $\mf{g}\boxtimes B$ be the pullback
$$\xymatrix{
\mf{g}\boxtimes B\ar[d]_\rho\ar[r] & \mf{g}\otimes B\ar[d]\\
T_A\ar[r] & T_A\otimes B.
}$$
All maps in this diagram are $A$-linear and preserve $L_\infty$-structures, and one can verify that the induced $L_\infty$-structure on $\mf{g}\boxtimes B$ turns it into an $L_\infty$-algebroid over $A$ (see \cite{vez13}). We therefore obtain a functor
$$\xymatrix{
\mf{g}\boxtimes (-)\colon \dgCAlg\ar[r] & \LooAlgd_A
}$$
which preserves pullbacks and fibrations, and weak equivalences when $\mf{g}$ is fibrant. Furthermore, there are natural isomorphisms $\mf{g}\boxtimes (B\otimes_k C) \cong \big(\mf{g}\boxtimes B\big)\boxtimes C$.

For any finite simplicial set $K$, let $\mf{g}^K=\mf{g}\boxtimes \mm{\Omega}[K]$ be the dg-Lie algebroid obtained by applying this functor to the polynomial differential forms on $K$.
\end{construction}
\begin{lemma}\label{la:lem:rightsimplicial}
Let $K\rt L$ be a cofibration between finite simplicial sets and let $\mf{g}\rt \mf{h}$ be a fibration. Then $\mf{g}^L\rt \mf{g}^{K}\times_{\mf{h}^K} \mf{h}^L$ is a fibration, which is a weak equivalence if $\mf{g}\rt \mf{h}$ or $K\rt L$ is a weak equivalence. 
\end{lemma}
\begin{proof}
It is well-known that the map $\mf{g}\otimes \mm{\Omega}[L]\rt \mf{g}\otimes \mm{\Omega}[K]\times_{\mf{h}\otimes \mm{\Omega}[K]} \mf{h}\otimes \mm{\Omega}[L]$
is a surjection and a quasi-isomorphism whenever $\mf{g}\rt \mf{h}$ or $K\rt L$ is a weak equivalence (cf.\ \cite{bou76}). The assertion now follows by considering the two pullback squares
$$\xymatrix{
\mf{g}^L\ar[r]\ar[d] &  \mf{g}^{K}\times_{\mf{h}^K} \mf{h}^L\ar[d]\ar[r] & T_A\ar[d]\\
\mf{g}\otimes \mm{\Omega}[L]\ar[r] & \mf{g}\otimes \mm{\Omega}[K]\times_{\mf{h}\otimes \mm{\Omega}[K]} \mf{h}\otimes \mm{\Omega}[L]\ar[r] & T_A\otimes \mm{\Omega}[L]
}$$
and using that (acyclic) fibrations are stable under base change.
\end{proof}
\begin{corollary}\label{la:cor:mappingspacesbetweenliealgebroids}
Let $\mf{g}$ be an $A$-cofibrant $L_\infty$-algebroid and let $\mf{h}$ be fibrant. Then the simplicial set \begin{equation}\label{eq:mappingspace}
\Map_{\LooAlgd_A^{(\infty)}}(\mf{g}, \mf{h}^{\Del[-]})
\end{equation}
is a model for the derived mapping space $\Map^\mathbb{R}(\mf{g}, \mf{h})$. 
\end{corollary}
\begin{remark}
When the $L_\infty$-algebroids $\mf{g}$ and $\mf{h}$ are concentrated in nonnegative degrees, one can also compute the mapping space using a semi-model structure on connective dg-Lie algebroids. In this case, one just has to assume that the map $\mf{h}\rt \tau_{\sgeq 0}T_A$ is a surjection in degrees $>0$; this becomes particularly easy when $A$ is discrete (so that $T_A$ is concentrated in degree $0$).
\end{remark}
\begin{remark}\label{rem:modelforloc}
The simplicial sets \eqref{eq:mappingspace} endow the category $\LooAlgd_A^{(\infty)}$ with an enrichment over simplicial sets. Corollary \ref{la:cor:mappingspacesbetweenliealgebroids} shows that the $\infty$-categorical localisation of the semi-model category of $L_\infty$-algebroids can be modelled by the full simplicial subcategory of $\LooAlgd_A^{(\infty)}$ on the $L_\infty$-algebroids which are fibrant and $A$-cofibrant.
\end{remark}
\begin{remark}
A similar analysis can be carried out with dg-Lie algebroids instead of $L_\infty$-algebroids: $\infty$-morphisms of dg-Lie algebroids $\mf{g}\leadsto \mf{h}$ correspond to maps out of a certain cofibrant replacement of $\mf{g}$, at least when $\mf{g}$ is $A$-cofibrant. This cofibrant replacement is simply the image of the $L_\infty$-algebroid $Q(\mf{g})$ under the left Quillen functor from Corollary \ref{la:cor:equivalencetohomotopyLRalgebras}. 

The $\infty$-categorical localisation of the semi-model category of dg-Lie algebroids is then the full simplicial subcategory of $\LieAlgd_A^{(\infty)}$ on the dg-Lie algebroids which are fibrant and $A$-cofibrant. In particular, Corollary \ref{la:cor:equivalencetohomotopyLRalgebras} shows that this simplicial category of dg-Lie algebroids is equivalent to the simplicial category of $L_\infty$-algebroids of Remark \ref{rem:modelforloc}.
\end{remark}

\section{Applications}\label{sec:appl}
In this section, we give two examples of the homological algebra that the semi-model structure on dg-Lie algebroids facilitates. First, we illustrate how classical Lie algebroid cohomology can be understood model-categorically in terms of mapping spaces. Furthermore, we provide a description of the \emph{loop space} of a dg-Lie algebroid $\mf{g}$. This loop space can naturally be considered as a dg-Lie algebroid in graded-mixed complexes, whose (graded-mixed) Lie algebroid cohomology can be described by the Weil algebra of $\mf{g}$.

\paragraph{Representations}
Recall that a \emph{representation} of a dg-Lie algebroid $\mf{g}$ over $A$ is given by a dg-$A$-module $E$, together with a Lie algebra representation
\begin{equation}\label{diag:strictrep}\vcenter{\xymatrix{
\nabla\colon \mf{g}\otimes_k E\ar[r] & E
}}\end{equation}
such that $\nabla_{ax} e = a\nabla_{x} e$ and $\nabla_x(ae) = x(a)e + (-1)^{ax} a \nabla_{x} e$, for all $a\in A$, $x\in \mf{g}$, and $e\in E$. There are at least two other ways of describing a $\mf{g}$-representation on a dg-$A$-module $E$:
\begin{enumerate}
\item Let $\mm{At}(E)$ be the Atiyah Lie algebroid of $E$, as in Example \ref{ex:atiyah}. Then a $\mf{g}$-representation is just a map of dg-Lie algebroids
$$\xymatrix{
\mf{g}\ar[r] & \mm{At}(E).
}$$
\item If $E$ is a $\mf{g}$-representation, then $\mf{g}\oplus E$ has the structure of a dg-Lie algebroid, with anchor map $(\rho, 0)\colon \mf{g}\oplus E\rt T_A$ and bracket 
$$
[(x, e), (y, f)] = ([x, y], \nabla_x f - \nabla_y e).
$$
There is an obvious inclusion and retraction $\mf{g}\rt \mf{g}\oplus E\rt \mf{g}$. Using these maps, one can realise the category of $\mf{g}$-representations as the full subcategory of $\mf{g}\big/\dgLieAlgd_A\big/\mf{g}$ on those retract diagrams $\mf{g}\rt \mf{g}\oplus \mf{m}\rt \mf{g}$ for which the Lie bracket vanishes on $\mf{m}\otimes\mf{m}$.
\end{enumerate}
\begin{remark}\label{rem:modulecategory}
Suppose that $\ope{P}$ is a dg-operad with $\ope{P}(0)=0$ and let $\ope{P}(1)$ be its dg-algebra of unary operations. Since $\ope{P}(1)$ is the quotient of the operad $\ope{P}$ by all morphisms of arity $\geq 2$, the category of dg-modules over $\ope{P}(1)$ is equivalent to category of dg-$\ope{P}$-algebras on which the operations of arity $\geq 2$ are all zero. On the other hand, the category of dg-$\ope{P}(1)$-modules is also equivalent to the category of \emph{abelian group objects} in the category of (all) dg-$\ope{P}$-algebras \cite[Lemma 1.3, 1.4]{ber09}.

Now recall that the category $\mf{g}\big/\dgLieAlgd_A\big/\mf{g}$ can be identified with the category of algebras over the reduced enveloping operad $\rEnv_\mf{g}$ of $\mf{g}$, via the assignment $\mf{g}\oplus \mf{m}\mapsto \mf{m}$ (Definition \ref{la:def:reducedenvelopingoperad}). Applying the above observations to algebras over $\rEnv_\mf{g}$, the following categories are then equivalent:
\begin{itemize}
\item the category of modules over $\rEnv_\mf{g}(1)$, the dg-algebra of unary operations in the reduced enveloping operad.
\item the category of abelian group objects in $\mf{g}\big/\dgLieAlgd_A\big/\mf{g}$.
\item the full subcategory of $\mf{g}\big/\dgLieAlgd_A\big/\mf{g}$ on the $\mf{g}\oplus \mf{m}$ such that the Lie bracket vanishes on $\mf{m}\otimes\mf{m}$. This uses that the 2-sided ideal of operations of arity $\geq 2$ is generated by the Lie bracket, by Remark \ref{la:rem:presentationofreducedenveloping}.
\end{itemize}
Unravelling the definitions, one sees that $\rEnv_\mf{g}(1)$ agrees with the usual \emph{enveloping algebra} $\mc{U}(\mf{g})$ of $\mf{g}$, as described in e.g.\ \cite{rin63}. Since the category $\Rep_\mf{g}^{\dg}$ of $\mf{g}$-representations is equivalent to the category of left modules over a dg-algebra, it carries a model structure with weak equivalences (fibrations) the quasi-isomorphisms (surjections).
\end{remark}
\begin{example}
Every dg-Lie algebroid has a natural representation on $A$ (via the anchor) and on the kernel of its anchor map (via the Lie bracket). 
\end{example}
\begin{example}\label{ex:lem:monoidalstructureonmodules}
The category $\dgRep_\mf{g}$ has a closed symmetric monoidal structure, given by $E\otimes_A F$ endowed with the $\mf{g}$-representation
$$
\nabla_x(e\otimes f) = \nabla_x(e)\otimes f + (-1)^{xe} e\otimes \nabla_x(f).
$$
The internal hom is given by $\Hom_A(E, F)$, equipped with the conjugate representation of $\mf{g}$.
This does not make $\dgRep_\mf{g}$ a monoidal model category, but for every $\mf{g}$-representation $E$ whose underlying dg-$A$-module is cofibrant, the functor $E\otimes_A (-)$ does preserve quasi-isomorphisms. 

When $\mf{g}$ is a cofibrant dg-Lie algebroid, the enveloping algebra $\mc{U}(\mf{g})\cong \Env_\mf{g}(1)$ is cofibrant by Theorem \ref{la:thm:monadicity}. It follows that every cofibrant $\mf{g}$-representation has a cofibrant underlying dg-$A$-module, so that the tensor product can be derived.
\end{example}
The above definitions have analogues for $L_\infty$-algebras:
\begin{definition}\label{la:def:moduleoverlooalgebroid}
Let $\mf{g}$ be an $L_\infty$-algebroid over $A$. A \emph{$\mf{g}$-representation} is a dg-$A$-module $E$, together with operations 
$$\xymatrix{
[x_1, ..., x_n, -]\colon E\ar[r] & E
}$$
of degree $|x_1|+...+|x_n|+n-2$ for every $x_1, ..., x_n\in \mf{g}$, such that (ignoring all Koszul signs due to permutations of variables)
\begin{align}\label{la:eq:repwithoutdiff1}
\begin{aligned}\leavevmode
[x_{\sigma(1)}, ..., x_{\sigma(n)}, e] &= (-1)^\sigma [x_1, ..., x_n, e] & \sigma\in \Sigma_n\\
[a \cdot x_1, ..., x_n, e] &= (-1)^{(n-1)a} a\cdot [x_1, ..., x_n, e]\\
[x_1, ..., x_n, a\cdot e] &= (-1)^{(n-1)a} a\cdot [x_1, ..., x_n, e] & n\geq 2\\
[x_1, a\cdot e] &= a\cdot [x_1, e] + x_1(a)\cdot e.
\end{aligned}
\end{align}
Furthermore, the brackets have to determine the structure of a module over the $L_\infty$-algebra $\mf{g}$, i.e.\
\begin{equation}\label{la:eq:looaction}
J^{n+1}(x_1, \dots, x_n, e)=0.
\end{equation}
for all $n\geq 0$, where $J^{n+1}$ is the Jacobiator from \ref{eq:jacobiator}, $x_i\in\mf{g}$ and $e\in E$.
\end{definition}
\begin{remark}\label{rem:repsoverliealgd}
When $\mf{g}$ is a dg-Lie algebroid, there are now two notions of $\mf{g}$-representation. To avoid confusion, we will call a representation in the sense of \eqref{diag:strictrep} a \emph{strict} $\mf{g}$-representation and a representation in the sense of Definition \ref{la:def:moduleoverlooalgebroid} will be called an \emph{$L_\infty$-representation} of $\mf{g}$.
\end{remark}
\begin{lemma}\label{lem:looreps}
Let $\mf{g}$ be an $L_\infty$-algebroid over $A$ and let $E$ be a dg-$A$-module. Then the following data is equivalent:
\begin{enumerate}
\item[(0)] an $L_\infty$-representation of $\mf{g}$ on $E$.
\item the structure of a retract diagram of $L_\infty$-algebroids on $\mf{g}\rt \mf{g}\oplus E\rt \mf{g}$, for which all brackets vanish when evaluated on at least two element of $E$.
\item an $\infty$-morphism $\mf{g}\leadsto \mm{At}(E)$ to the Atiyah Lie algebroid of $E$.
\end{enumerate}
\end{lemma}
\begin{proof}
Unwinding the definitions, the nontrivial brackets on $\mf{g}\oplus E$ precisely correspond to the operations from Definition \ref{la:def:moduleoverlooalgebroid}. This shows that (0) and (1) are equivalent.

By \cite{hin93}, a $k$-linear $L_\infty$-algebra representation of $\mf{g}$ on $E$ is equivalent to the data of a twisting cochain $\tau\colon \ol{C}_*(\mf{g})\rt \mm{End}_k(E)[1]$ to the endomorphism Lie algebra of $E$: the map $\tau$ is given by
$$\xymatrix{
\mm{Sym}^n_k\mf{g}[1]\ar[r] & \mm{End}_k(E)[1]; \hspace{4pt} x_1\otimes \dots\otimes x_n\ar[r] & [x_1, \dots, x_n, -].
}$$
The conditions \eqref{la:eq:repwithoutdiff1} are now equivalent to the condition that 
$$\xymatrix{
(\rho, \tau)\colon \ol{C}_*(\mf{g})\ar[r] & \big(T_A\oplus \mm{End}_k(E)\big)[1]
}$$
takes values in the Atiyah Lie algebroid of $E$ and is graded $A$-linear.
\end{proof}
By (1), the category $\Rep^{(\infty)}_\mf{g}$ of $L_\infty$-representations of an $L_\infty$-algebroid $\mf{g}$ is equivalent to the category of modules over $\rEnv_\mf{g}(1)$, the unary operations of its reduced enveloping operad. In particular, the category  carries a model structure in which a map is a weak equivalence (fibration) if and only if the underlying map of dg-$A$-modules is one.
\begin{corollary}
A map $f\colon \mf{g}\rt \mf{h}$ of $L_\infty$-algebroids induces a Quillen adjunction
$$\xymatrix{
f_*\colon \Rep^{(\infty)}_\mf{g}\ar@<1ex>[r] & \Rep^{(\infty)}_\mf{h}\ar@<1ex>[l]\colon f^!.
}$$
This is a Quillen equivalence whenever $f$ is a weak equivalence between $A$-cofibrant $L_\infty$-algebroids.
\end{corollary}
\begin{proof}
The Quillen pair $(f_*, f^!)$ is induced by induction and restriction along the map of reduced enveloping operads $f\colon \rEnv_\mf{g}\rt \rEnv_\mf{h}$. For the second assertion, let $Q(f)\colon Q(\mf{g})\rt Q(\mf{h})$ denote the map induced by $f$ on `cobar' resolutions of Lemma \ref{lem:Qdescribed}. This is a weak equivalence between cofibrant $L_\infty$-algebroids. Applying the left Quillen equivalence of Corollary \ref{la:cor:equivalencetohomotopyLRalgebras} to $Q(f)$, we obtain a weak equivalence between cofibrant dg-Lie algebroids that we will denote by $\tilde{f}\colon \tilde{\mf{g}}\rt \tilde{\mf{h}}$.

By part (2) of Lemma \ref{lem:looreps}, an $L_\infty$-representation of the $L_\infty$-algebroid $\mf{g}$ is just a map of dg-Lie algebroids $\tilde{\mf{g}}\rt \mm{At}(E)$. It follows that there is a natural equivalence
$$
\Rep^{(\infty)}_\mf{g}\simeq \Rep^{\dg}_{\tilde{\mf{g}}}
$$
between the category of $L_\infty$-representations of $\mf{g}$ and the category of strict representations of the cofibrant dg-Lie algebroid $\tilde{\mf{g}}$. This equivalence of categories identifies the model structures on both sides, which are both transferred from $\dgMod_A$. The Quillen pair $(f_*, f^!)$ can now be identified with the Quillen pair
$$\xymatrix{
\tilde{f}_*\colon \Rep^{\dg}_{\tilde{\mf{g}}}\ar@<1ex>[r] & \Rep^{\dg}_{\tilde{\mf{h}}}\ar@<1ex>[l]\colon \tilde{f}^!.
}$$
associated to the map of dg-Lie algebroids $\tilde{f}\colon \tilde{\mf{g}}\rt \tilde{\mf{h}}$. By Lemma \ref{lem:hoinvarianceofenveloping}, a weak equivalence between cofibrant dg-Lie algebroids induces a weak equivalence on enveloping operads, so that the above Quillen pair is a Quillen equivalence.
\end{proof}

\paragraph{Lie algebroid cohomology}
Let $\mf{g}$ be an $L_\infty$-algebroid and let $E$ be a representation of $\mf{g}$. Our aim is to give a cohomological description of the derived space of sections
$$\xymatrix{
\mf{g}\ar[rd]_{=} \ar@{..>}[r] & \mf{g}\oplus E\ar[d]^{\pi}\\
& \mf{g}
}$$
of the canonical projection from $\mf{g}\oplus E$ to $\mf{g}$. 

Because the semi-model structure on $L_\infty$-algebroids is right proper, the derived space of sections of $\pi$ can be computed as the derived mapping space from $\mf{g}$ to $\mf{g}\oplus E$ in the slice semi-model category $\dgLooAlgd_A/\mf{g}$. As an object over $\mf{g}$, the $L_\infty$-algebroid $\mf{g}\oplus E$ is fibrant. Furthermore, it admits a simple fibrant simplicial resolution, given by
$$\xymatrix{
\Del^\op\ar[r] & \dgLooAlgd_A/\mf{g}; \hspace{5pt} [n]\ar@{|->}[r] & \mf{g}\oplus C^*(\Del[n], E).
}$$
Here $C^*(\Del[n], E)$ are the normalised cochains on $\Del[n]$ with coefficients in $E$, which carry a natural $\mf{g}$-representation.

Let us assume from now on that $\mf{g}$ is an $A$-cofibrant $L_\infty$-algebroid, so that it has an explicit cofibrant replacement $q\colon Q(\mf{g})\rt \mf{g}$, as described in Section \ref{la:sec:cofibrantreplacement}. The datum of map of $L_\infty$-algebroids 
$s\colon Q(\mf{g})\rt \mf{g}\oplus E$ over $\mf{g}$ is equivalent to the datum of a graded $A$-linear map 
$$\xymatrix{
(q, \alpha)\colon \mm{Sym}^{\sgeq 1}_A\mf{g}[1]\ar[r] & \mf{g}[1]\oplus E[1]
}$$
satisfying the Maurer-Cartan equation \eqref{la:eq:maurercartan}, where $q$ is the obvious projection onto $\mf{g}[1]$. Since $q$ already satisfies the Maurer-Cartan equation and there are no nontrivial brackets between elements in $E$, the Maurer-Cartan equation reduces to the following (linear) equation for $\alpha$:
\begin{equation}\label{fm:eq:maurercartanmodule}
\dau_\mm{CE}(\alpha):=\dau_E\circ\alpha - \alpha\circ \dau_{\ol{C}_*(\mf{g})} + \sum_{k\geq 1}\frac{1}{k!} [q, q, \dots, q, \alpha]_{k+1} \stackrel{!}{=} 0.
\end{equation}
\begin{lemma}\label{lem:cecomplex}
Formula \eqref{fm:eq:maurercartanmodule} determines a differential on the graded vector space of graded $A$-linear maps
$$
\ul{\Hom}_A\big(\mm{Sym}^{\sgeq 1}_A\mf{g}[1], E[1]\big).
$$
\end{lemma}
\begin{proof}
Formula \eqref{fm:eq:maurercartanmodule} determines an $\mathbb{R}$-linear map of (homological) degree $-1$
$$\xymatrix{
\dau_\mm{CE}\colon \ul{\Hom}_A(\mm{Sym}_A\mf{g}[1], E[1])\ar[r] & \ul{\Hom}_{\mathbb{R}}(\mm{Sym}_\mathbb{R}\mf{g}[1], E[1])
}$$ 
To verify that this preserves $A$-multilinear maps and squares to zero, replace $E$ by its cone $E[0, 1]$. Unravelling the definition, a map $Q(\mf{g})\rt \mf{g}\oplus E[0, 1]$ over $\mf{g}$ is determined by a pair of maps
$$\xymatrix{
\alpha\colon \mm{Sym}^{\sgeq 1}_A\mf{g}[1]\ar[r] & E[1] & \beta\colon \mm{Sym}^{\sgeq 1}_A\mf{g}[1]\ar[r] & E[2]
}$$
subject to the condition that $\dau_\mm{CE}\alpha = \beta$ and $\dau_\mm{CE}\beta=0$. 

On the other hand, a map $Q(\mf{g})\rt \mf{g}\oplus E[0, 1]$ over $\mf{g}$ is determined uniquely by a map $Q(\mf{g})\rt \mf{g}\oplus E$ of $L_\infty$-algebroids without differential over $\mf{g}$. Without differential, $Q(\mf{g})$ is freely generated by the graded $A$-module $\big(\mm{Sym}^{\sgeq 1}_A\mf{g}[1]\big)[-1]$ and the map $Q(\mf{g})\rt \mf{g}\oplus E$ is classified by the map $\alpha$ above. 

It follows that for every graded $A$-linear $\alpha$, there is a unique graded $A$-linear $\beta$ such that $\beta=\dau_\mm{CE}\alpha$ and $\dau_\mm{CE}\beta=0$. It follows that $\dau_\mm{CE}$ preserves graded $A$-linear maps and squares to zero.
\end{proof}
\begin{definition}\label{fm:def:reducedcecomplex}
Let $\ol{C}{}^*(\mf{g}, E)$ be the chain complex $\Hom_A\big(\mm{Sym}_A\mf{g}[1], E\big)$, equipped with the differential $\dau_\mm{CE}$ given by formula \ref{fm:eq:maurercartanmodule}. We will refer to $\ol{C}{}^*(\mf{g}, E)$ as the \emph{reduced Chevalley-Eilenberg complex} of $\mf{g}$ with coefficients in $E$.
\end{definition}
The Chevalley-Eilenberg differential $\dau_\mm{CE}$ can be computed explicitly, using that the value of $\frac{1}{k!}[\pi, \pi, \dots, \pi, \alpha]_{k+1}$ on an element $X_1\dots X_n\in \mm{Sym}_A^{\sgeq 1}\mf{g}[1]$ is given by the sum (ignoring Koszul signs)
$$
\sum_{\sigma\in \mm{UnSh}(k, n-k)} \Big[x_{\sigma(1)}, \cdots, x_{\sigma(k)}, \alpha(x_{\sigma(k+1)}, \dots, x_{\sigma(n)})\Big].
$$
Using this, one obtains (modulo Koszul signs) the explicit formula
\begin{align}\label{fm:eq:cediff}
 (\dau_\mm{CE}\alpha)(x_1, ..., x_n) & =  \sum_{k\geq 1}\sum_{\sigma\in \mm{UnSh}(k, n-k)} \hspace{2pt} \big[x_{\sigma(1)}, ..., x_{\sigma(k)}, \alpha(x_{\sigma(k+1)}, ..., x_{\sigma(n)})\big]\nonumber\\
&- \sum_{k\geq 1}\sum_{\sigma\in \mm{UnSh}(k, n-k)} \hspace{2pt} \alpha\Big([x_{\sigma(1)}, ..., x_{\sigma(k)}], x_{\sigma(k+1)}, ..., x_{\sigma(n)}\Big)\nonumber\\
&+ \dau_E\big(\alpha(x_1, ..., x_n)\big).
\end{align}
This is precisely the formula for the usual Chevalley-Eilenberg (or de Rham) differential on
$$
C^*(\mf{g}, E) = \ul{\Hom}_A\big(\mm{Sym}_A\mf{g}[1], E\big)
$$
which computes the cohomology of the $L_\infty$-algebroid $\mf{g}$ with coefficients in $E$. The reduced Chevalley-Eilenberg complex is simply the kernel of the canonical map of chain complexes $C^*(\mf{g}, E)\rt E$ evaluating at $1\in \mm{Sym}_A\mf{g}[1]$. Consequently, one can think of its homotopy groups as the \emph{reduced} cohomology groups $\ol{\mm{H}} {}^i(\mf{g}, E)$ of $\mf{g}$ with coefficients in $E$.

\begin{definition}
For a chain complex $V$ over $k$, let us write $\mm{\Omega}^\infty(V):= \Map^\mathbb{R}(k,V)$ for the derived mapping space from $k$ to $V$. This notation is justified by the following: the $\infty$-category of chain complexes over $k$ has a compact generator $k$ and is therefore equivalent to the $\infty$-category of $\mm{H}(k)$-module spectra \cite[Theorem 7.1.2.1]{lur16}. The space $\mm{\Omega}^\infty(V)$ is exactly the (infinite loop) space underlying the $\mm{H}(k)$-module spectrum associated to $V$.
\end{definition}
\begin{corollary}\label{cor:cohomologyashomotopy}
Let $\mf{g}$ be an $A$-cofibrant $L_\infty$-algebroid and let $E$ be a representation of $\mf{g}$. There is a (natural) equivalence
$$
\Map^\mathbb{R}_{/\mf{g}}(\mf{g}, \mf{g}\oplus E) \simeq \mm{\Omega}^\infty \ol{C}{}^*(\mf{g}, E[1])
$$
between the derived space of sections $\mf{g}\rt \mf{g}\oplus E$ and the space associated to the reduced Chevalley-Eilenberg complex $\ol{C}{}^*(\mf{g}, E[1])$. In particular, there is a natural bijection
$$
\pi_0\Map^\mathbb{R}_{/\mf{g}}(\mf{g}, \mf{g}\oplus E)\cong \ol{\mm{H}} {}^1(\mf{g}, E).
$$
\end{corollary}
\begin{proof}
By Lemma \ref{lem:cecomplex} and the discussion preceding it, an explicit model for the derived mapping space $\Map^\mathbb{R}_{/\mf{g}}(\mf{g}, \mf{g}\oplus E)$ is given by the simplicial set
$$\xymatrix{
\Del^\op\ar[r] & \Set; \hspace{4pt} [n]\ar@{|->}[r] & Z_0\ol{C}{}^*\Big(\mf{g}, C^*(\Del[n], E[1])\Big).
}$$
There are natural isomorphisms
$$
\ol{C}{}^*\Big(\mf{g}, C^*(\Del[n], E[1])\Big)\cong C^*\Big(\Del[n], \ol{C}^*(\mf{g}, E[1])\Big).
$$
The simplicial set of sections of $\mf{g}\oplus E\rt \mf{g}$ can therefore be identified with the simplicial set of $k$-linear maps $k\rt C^*\big(\Del[-], \ol{C}^*(\mf{g}, E[1])\big)$. The latter precisely computes the derived mapping space from $k$ to $\ol{C}^*(\mf{g}, E[1])$, i.e.\ the infinite loop space $\mm{\Omega}^{\infty}\ol{C}{}^*(\mf{g}, E[1])$.
\end{proof}
\begin{example}
Recall that there is a canonical representation of $\mf{g}$ on $A$ via its anchor map. By shifting degrees, we also obtain representations of $\mf{g}$ on the various $A[n]$. By Corollary \ref{cor:cohomologyashomotopy}, we have that
$$
\ol{\mm{H}} {}^n(\mf{g}, A)\cong \ol{\mm{H}} {}^1(\mf{g}, A[n-1])\cong \pi_0\Map^\mathbb{R}_{/\mf{g}}(\mf{g}, \mf{g}\oplus A[n-1]).
$$ 
Note that $\mf{g}\oplus A[n-1]$ is isomorphic to the fibre product $\mf{g}\times_{T_A} T_A\oplus A[n-1]$. Since $T_A\oplus A[n-1]\rt T_A$ is a fibration, this pullback is a homotopy pullback, so that
$$
\Map^\mathbb{R}_{/\mf{g}}(\mf{g}, \mf{g}\oplus A[n-1])\simeq \Map^\mathbb{R}_{/T_A}(\mf{g}, T_A\oplus A[n-1])\simeq \Map^\mathbb{R}(\mf{g}, T_A\oplus A[n-1]).
$$
The last equivalence uses that $T_A$ is the \emph{terminal} $L_\infty$-algebroid. We conclude that the reduced cohomology of $\mf{g}$ with values in its canonical representation can be identified with homotopy classes of maps of $L_\infty$-algebroids
$$
\ol{\mm{H}} {}^n(\mf{g}, A)\cong \pi_0\Map^\mathbb{R}\big(\mf{g}, T_A\oplus A[n-1]\big).
$$
\end{example}
\begin{example}\label{la:ex:loofromcocycle}
Let $\mf{g}$ be an $A$-cofibrant $L_\infty$-algebroid and let 
$$
\alpha\in \ol{C}{}^*(\mf{g}, E)=\Hom_A(\mm{Sym}^{\sgeq 1}_A\mf{g}[1], E)
$$
be a degree $0$ cycle in the reduced Chevalley-Eilenberg complex with coefficients in an $L_\infty$-representation $E$ of $\mf{g}$. Associated to $\alpha$ is a map of $L_\infty$-algebroids $\mf{g}_\alpha\rt \mf{g}$, given by the projection $\mf{g}\oplus E[-2]\rt \mf{g}$ at the level of graded $A$-modules. The bracket on $\mf{g}_\alpha$ is given as follows: a bracket of multiple elements in $E[-2]$ is zero and the bracket of elements in $\mf{g}$ with an element in $E[-2]$ is given by the $\mf{g}$-representation on $E$. The bracket and differential of elements in $\mf{g}$ is given by
\begin{equation}\label{eq:alphatwistedbracket}
[x_1, \dots, x_k] = \big([x_1, \dots, x_k]_\mf{g}, \alpha(x_1, \dots, x_k)\big) \qquad\quad \dau(x) = (\dau_\mf{g}(x), \alpha(x)).
\end{equation}
We claim that there are pullback squares of $L_\infty$-algebroids (and strict maps)
\begin{equation}\label{diag:zigzagpullbacks}\vcenter{\xymatrix@R=1.9pc{
\mf{g}_\alpha\ar[d] & Q(\mf{g})_\alpha\ar[d]\ar[l]_-\sim\ar[r] & \mf{g}\oplus E[0, 1]\ar[d]\\
\mf{g} & Q(\mf{g})\ar[l]_-\sim^-q \ar[r]_-{(q, \alpha)} & \mf{g}\oplus E[1].
}}\end{equation}
Here $q$ is the canonical weak equivalence from the cobar resolution of $\mf{g}$ and $Q(\mf{g})_\alpha$ is \emph{defined} by the left pullback square. Note that $Q(\mf{g})_\alpha\rt \mf{g}_\alpha$ is the pullback of a weak equivalence along a fibration, and hence a weak equivalence itself. 

Indeed, as graded $A$-modules, there is an isomorphism $Q(\mf{g})_\alpha\cong Q(\mf{g})\oplus E$, but the brackets and differential are twisted by $\alpha$ as in Equation \eqref{eq:alphatwistedbracket}. Now recall that without differential, $Q(\mf{g})$ is freely generated by $\big(\mm{Sym}_A^{\sgeq 1}\mf{g}[1]\big)[-1]$. Consequently, the obvious linear section 
$$\xymatrix{
(\mm{id}, 0)\colon \big(\mm{Sym}_A^{\sgeq 1}\mf{g}[1]\big)[-1] \ar[r] & Q(\mf{g})\oplus E\cong Q(\mf{g})_\alpha
}$$
extends to a splitting $Q(\mf{g})\rt Q(\mf{g})_\alpha\rt Q(\mf{g})$, where both maps preserve the brackets. This splitting induces a \emph{different} isomorphism $Q(\mf{g})_\alpha\cong' Q(\mf{g})\oplus E$. Without differential, this isomorphism $\cong'$ identifies $Q(\mf{g})_\alpha$ with the square zero extension of $Q(\mf{g})$ by the $Q(\mf{g})$-representation $E$ (obtained by restriction along $Q(\mf{g})\rt \mf{g})$. A tedious but straightforward computation, using the description of the differential on $Q(\mf{g})$ from Lemma \ref{lem:Qdescribed}, then shows that the differential on $Q(\mf{g})_\alpha$ takes the form
$$
\dau(v, e) = \big(\dau_{Q(\mf{g})}(v), \dau_E(e)+\alpha(v)\big) \qquad\qquad (v, e)\in Q(\mf{g})\oplus E\cong' Q(\mf{g})_\alpha.
$$
Finally, the pullback of the right square is exactly the square zero extension $Q(\mf{g})\oplus E$ endowed with the above differential.

Note that the map $\mf{g}\oplus E[0,1]\rt \mf{g}\oplus E[1]$ is weakly equivalent to the map $(\mm{id}, 0)\colon \mf{g}\rt \mf{g}\oplus E[1]$. At the $\infty$-categorical level, the diagram \eqref{diag:zigzagpullbacks} therefore exhibits $\mf{g}_\alpha$ as the homotopy pullback of the map $\alpha\colon \mf{g}\leadsto \mf{g}\oplus E[1]$ and the map $(\mm{id}, 0)\colon \mf{g}\rt \mf{g}\oplus E[1]$.
\end{example}

\paragraph{Loop spaces}
As in every model category, the \emph{free loop space} $\mc{L}\mf{g}$ of an $L_\infty$-algebroid is the homotopy limit of the constant $S^1$-diagram with value $\mf{g}$. There are various ways to compute such free loop spaces, depending on a choice of cell decomposition for the circle $S^1$. Taking the CW-structure on $S^1$ with two zero cells and two 1-cells, one obtains the usual description of $\mc{L}\mf{g}$ as the derived self-intersection of the diagonal map $\mf{g}\rt \mf{g}\times^h_{T_A} \mf{g}$.

When $\mf{g}$ is a \emph{fibrant} (i.e.\ transitive) $L_\infty$-algebroid, there is a simple way to compute $\mc{L}\mf{g}$ using the cotensoring of $L_\infty$-algebroids over (unbounded) cdgas from Construction \ref{la:cons:simplicialcotensor}. Indeed, $\mc{L}\mf{g}$ can simply be computed as $\mf{g}\boxtimes \mm{\Omega}[S^1]$, where $S^1$ is some finite simplicial model for the circle and $\mm{\Omega}[S^1]$ is the cdga of polynomial differential forms on it. In fact, recall that $H^*(S^1)=k[\epsilon_{-1}]$ is the free graded algebra on a generator of (homological) degree $-1$. We can therefore choose a weak equivalence of cdgas $k[\epsilon_{-1}]\rt \mm{\Omega}[S^1]$ and identify
$$
\mc{L}\mf{g} = \mf{g}\boxtimes k[\epsilon_{-1}].
$$
Let $\mf{n}$ denote the kernel of the anchor map $\rho\colon \mf{g}\rt T_A$, which carries a natural (adjoint) $\mf{g}$-representation. Unwinding the definitions, we find that $\mf{g}\boxtimes k[\epsilon_{-1}]$ is isomorphic to $\mf{g}\oplus \mf{n}[-1]$. There are maps
$$\xymatrix{
\mf{g}\ar[r]^-\iota & \mc{L}\mf{g}=\mf{g}\oplus \mf{n}[-1]\ar[r]^-{\pi} & \mf{g}
}$$
which realise $\mf{g}\oplus \mf{n}[-1]$ as a square zero extension of $\mf{g}$. The inclusion $\iota$ is the canonical map induced by $k\rt \mm{\Omega}[S^1]$, but the projection $\pi$ is not canonical: it depends on a choice of basepoint for $S^1$.
\begin{lemma}\label{lem:gradedmixedstructure}
Let $\mf{g}$ be a fibrant $L_\infty$-algebroid. Then $\mc{L}\mf{g}= \mf{g}\boxtimes k[\epsilon_{-1}]$ has the structure of a graded-mixed $L_\infty$-algebroid (Variant \ref{var:gradedmixed}), whose underlying graded-mixed complex is given by
$$
\mc{L}\mf{g}(-1) = \mf{n}[-1] \qquad \qquad \mc{L}\mf{g}(0)=\mf{g}
$$
with $d\colon \mc{L}\mf{g}(-1)\rt \mc{L}\mf{g}(0)[-1]$ the obvious inclusion.
\end{lemma}
\begin{proof}
The bracket of an element in $\mc{L}\mf{g}(p)$ with elements in $\mc{L}\mf{g}(0)$ is again contained in $\mc{L}\mf{g}(p)$ and the bracket of at least two elements in $\mc{L}\mf{g}(-1)$ is zero. It follows that the $L_\infty$-algebroid structure is compatible with the grading. It is compatible with the mixed structure because for any $x_i\in\mf{g}, \xi\in \mf{n}[-1]$, we have that
$$
d[x_1, \dots, x_n, \xi] = [x_1, \dots, x_n, d\xi]
$$
is simply given by the $(n+1)$-fold bracket in $\mf{g}$.
\end{proof}
\begin{remark}
The graded-mixed structure on $\mc{L}\mf{g}$ can be understood as follows. Recall (see e.g.\ \cite{ben12}, \cite{pan13}) that graded mixed complexes can be viewed as dg-comodules over the Hopf algebra
$$
\mc{H} := k[t, t^{-1}, \epsilon_{-1}] \qquad \Del(t) = t\otimes t\qquad \Del(\epsilon_{-1}) = \epsilon_{-1}\otimes t.
$$
Here $t$ has degree $0$ and $\epsilon_{-1}$ has (homological) degree $-1$. There is a coaction
$$\xymatrix{
k[\epsilon_{-1}] \ar[r] & k[\epsilon_{-1}]\otimes_k \mc{H}; \hspace{4pt}\epsilon_{-1}\ar@{|->}[r] & \epsilon_{-1}\otimes t
}$$
which induces a coaction of $\mc{H}$ on the free loop space $\mc{L}\mf{g}$ of the form
$$\xymatrix{
\mc{L}\mf{g} = \mf{g}\boxtimes k[\epsilon_{-1}] \ar[r] & \mf{g}\boxtimes \big(k[\epsilon_{-1}]\otimes_k \mc{H}\big)\cong \mc{L}\mf{g}\boxtimes \mc{H},
}$$
simply by restricting the canonical $\mc{H}$-comodule structure on $\mf{g}\otimes k[\epsilon_{-1}]$. Unwinding the definitions, this coaction of $\mc{H}$ on $\mc{L}\mf{g}$ corresponds to the graded mixed structure on $\mc{L}\mf{g}$ described in Lemma \ref{lem:gradedmixedstructure}.
\end{remark}
\begin{remark}\label{rem:loops}
The mixed structure on $\mc{L}\mf{g}$ can also be interpreted homotopy-theoretically at follows. The sub-Hopf algebra $k[\epsilon_{-1}]\subseteq \mc{H}$ can be identified with the cohomology ring $H^*(S^1)$, which inherits a Hopf algebra structure from the group multiplication $\mu\colon S^1\times S^1\rt S^1$. The mixed structure on $\mc{L}\mf{g}$ is then encoded by the coaction
$$\xymatrix{
\mc{L}\mf{g} = \mf{g}\boxtimes H^*(S^1)\ar[rr]^-{\mf{g}\boxtimes H^*(\mu)} & & \mf{g}\boxtimes H^*(S^1\times S^1) \cong \mc{L}\mf{g}\boxtimes H^*(S^1).
 }$$ 
In fact, the Hopf algebra $k[\epsilon_{-1}]$ provides a rational model for $S^1$, together with its group structure \cite{toe11} (see also \cite[Remark 3.16]{ben12}). One can therefore think of the mixed structure on $\mc{L}\mf{g}$ as a (rational) algebraic incarnation of the $S^1$-action on $\mc{L}\mf{g}$ by rotation of loops. This more topological (rather than algebraic) perspective is used in derived geometry, cf.\ Remark \ref{rem:diffformsvsloops}.
\end{remark}
The Chevalley-Eilenberg complex of a graded mixed $L_\infty$-algebroid over $A$ can be computed internally to graded mixed complexes, using exactly the same formulas as in Definition \ref{fm:def:reducedcecomplex}. Applying this to the free loop space of an $L_\infty$-algebroid $\mf{g}$, we obtain the following:
\begin{example}\label{dth:con:weilalgebrasimple}
Let $\rho\colon \mf{g}\rt T_A$ be a fibrant $L_\infty$-algebroid and let $\mc{L}\mf{g}=\mf{g}\oplus \mf{n}[-1]$ be its free loop space. The Chevalley-Eilenberg complex $C^*(\mc{L}\mf{g})$ with coefficients in $A$ decomposes into graded pieces
$$
C^*(\mc{L}\mf{g})(q) := \ul{\Hom}{}_A\Big(\mm{Sym}_A\mf{g}[1]\otimes \mm{Sym}^q_A\mf{n}, A\Big)
$$
consisting of maps that are polynomial of degree $q$ in $\mf{n}$. This is indeed stable under the differential \eqref{fm:eq:cediff} because $\mf{n}[-1]\subseteq \mc{L}\mf{g}$ is square zero. Furthermore, $C^*(\mc{L}\mf{g})$ carries a mixed structure $d\colon C^*(\mc{L}\mf{g})(q)\rt C^*(\mc{L}\mf{g})(q+1)[-1]$ given by 
$$
(d\alpha)(x_1, \dots, x_n, \xi_1, \dots, \xi_q) = \sum_{j=1}^q \alpha(x_1, \dots, x_n, \sigma(\xi_j), \xi_{1}, \dots, \xi_q)
$$
where $x_i\in\mf{g}$, $\xi_j \in \mf{n}$ and $\sigma\colon \mf{n}\rt \mf{g}$ is the inclusion, which gave the mixed structure on $\mc{L}\mf{g}$. The resulting graded mixed complex $C^*(\mc{L}\mf{g})$ has the structure of a graded mixed cdga, with multiplication given by the usual product of forms.
\end{example}
\begin{remark}
The Chevalley-Eilenberg complex of graded mixed $L_\infty$-algebroids is only homotopy invariant when applied to $\mf{h}$ for which each $\mf{h}(p)$ is a cofibrant dg-$A$-module. For free loop spaces $\mc{L}\mf{g}$, this is the case when $\mf{g}\rt T_A$ is a fibration between cofibrant dg-$A$-modules. For the remainder of this section, we will therefore assume that \emph{$T_A$ is a cofibrant dg-$A$-module}. One can always replace $T_A$ by a weakly equivalent dg-Lie algebroid for which this is the case. 
\end{remark}
For sufficiently nice dg-Lie algebroids, the mixed graded complex $C^*(\mc{L}\mf{g})$ has a more traditional description as the \emph{Weil algebra} $W(\mf{g})$ of $\mf{g}$ \cite{aba11}. In fact, $W(\mf{g})$ can also be used to describe the cohomology of $\mc{L}\mf{g}$ when $\mf{g}$ is not fibrant.
\begin{construction}\label{dth:con:weilalgebra}
Suppose that $T_A$ is a cofibrant dg-$A$-module and consider the graded mixed complex
$$
\widehat{\dR}(A)(p) = \ul{\Hom}_A\Big(\mm{Sym}_A^p\big(T_A[-1]\big), A\Big)
$$
with mixed structure given by the de Rham differential (without Koszul signs)
$$
d_\dR\omega(v_1, \cdots, v_n) = \sum_i v_i\big(\omega(v_1, \dots, v_n)\big) - \sum_{i<j} \omega([v_i, v_j], v_1, \dots, v_n).
$$
The result is a graded mixed cdga $\widehat{\dR}(A)$, to which the usual Cartan calculus of differential forms can be applied: in addition to the de Rham differential, every $v\in T_A$ yields operators 
$$\xymatrix{
\iota_v\colon \widehat{\dR}(A)\ar[r] & \widehat{\dR}(A)[1] & \mc{L}_v = d\iota_v + \iota_v d\colon \widehat{\dR}(A)\ar[r] & \widehat{\dR}(A)
}$$
where $\iota_v$ is given by $(\iota_v\omega)(v_1, \cdots, v_n)= \omega(v, v_1, \dots, v_n)$.

Let $\rho\colon \mf{g}\rt T_A$ be a dg-Lie algebroid over $T_A$ and let $\mf{g}\oplus \mf{g}[-1]$ be the square zero extension of its underlying dg-Lie algebra (over $k$) by the shifted adjoint representation. There is a canonical representation of $\mf{g}\oplus \mf{g}[-1]$ on $\widehat{\dR}(A)$, where an element $(x, \xi)\in \mf{g}\oplus \mf{g}[-1]$ acts by the derivation
$$
(x, \xi)\cdot \omega = \mc{L}_{\rho(x)}\omega + \iota_{\rho(\xi)}\omega
$$
Consider the graded subalgebra of the associated Chevalley-Eilenberg complex consisting of maps $\mm{Sym}_k(\mf{g}[1]\oplus \mf{g})\rt \widehat{\dR}(A)$ that are $A$-multilinear in $\mf{g}[1]\oplus \mf{g}$
$$
\widehat{W}(\mf{g})\subseteq \mc{A}:=C^*\Big(\mf{g}\oplus \mf{g}[-1], \widehat{\dR}(A)\Big).
$$
This subalgebra is closed under the Chevalley-Eilenberg differential by \cite[Proposition 3.5]{aba11}, where it is called the vertical differential. In fact, unwinding the definitions yields the following description of $\widehat{W}(\mf{g})$. Let 
$$
\tilde{\mf{n}}=\mm{fib}(\mf{g}\rt T_A)=\mf{g}\oplus T_A[-1]
$$
denote the mapping fibre of the anchor map of $\mf{g}$. Then $\widehat{W}(\mf{g})$ factors as
$$
\widehat{W}(\mf{g})=\prod_q W^q(\mf{g}) = \prod_q \ul{\Hom}_A\Big(\mm{Sym}_A\mf{g}[1]\otimes_A \mm{Sym}^q\big(\tilde{\mf{n}}\big), A\Big).
$$
For $x_i\in\mf{g}[1], \xi_j\in \mf{g}$ and $v_k\in T_A[-1]$, the differential is given by
\begin{align*}
(\dau\alpha)(x, \xi, v) &= \dau_A\big(\alpha(x, \xi, v)\big) - \alpha\big(\dau_{\mm{lin}}(x, \xi, v)\big) + \sum_i x_i\cdot \alpha(\dots) \\
&-\sum_{i<i'} \alpha([x_i, x_{i'}], \dots)- \sum_{i, j}\alpha([x_i, \xi_j], \dots) -\sum_{i, k} \alpha([x_i, v_k], \dots).
\end{align*}
Here the dots indicate all variables $x_i, \xi_j$ and $v_k$ not appearing before. Furthermore, $\dau_{\mm{lin}}(x, \xi, v)$ is the ($A$-linear) differential on $\mm{Sym}_A(\mf{g}[1]\oplus \tilde{\mf{n}})$ induced by the differentials on $\mf{g}$ and $\tilde{\mf{n}}$. The brackets
$$
[x_i, x_{i'}]\in \mf{g}[1]\qquad\quad [x_i, \xi_j]\in \mf{g}\qquad\quad [x_i, v_k]\in T_A[-1]
$$
arise from the $k$-linear Lie algebra representation of $\mf{g}$ on itself and on the mapping fibre $\tilde{n}$. In particular, the Chevalley-Eilenberg differential preserves each factor $W^q(\mf{g})$. In addition, there are derivations $d\colon W^q(\mf{g})\rt W^{q+1}(\mf{g})[-1]$, given (modulo Koszul signs) by
$$
(d\alpha)(x, \xi, -) = d_\dR(\alpha(x, \xi, -)) + \sum_{j} \alpha(x, \sigma(\xi_j), \xi, -) \quad \in \widehat{\dR}(A).
$$
In the second term, $\sigma(\xi_j)\in \mf{g}[1]$ indicates the element $\xi_j$, considered as an element in $\mf{g}[1]$ instead of $\mf{g}$. In other words, we replace one of the variables $\xi_j$ by the corresponding `$x$-variable', leaving the rest of the variables the same. The map $d$ is a chain map which squares to zero, since the de Rham differential on $\widehat{\dR}(A)$ does. 
\end{construction}
\begin{definition}
The \emph{Weil algebra} of a dg-Lie algebroid $\mf{g}$ over $A$ is the graded mixed cdga from Construction \ref{dth:con:weilalgebra}
$$
W(\mf{g}) =\left(\bigoplus_q W^q(\mf{g})\subseteq \widehat{W}(\mf{g}), \dau, d\right).
$$
\end{definition}
\begin{lemma}\label{lem:weilishomotopyinvariant}
Suppose that $T_A$ is a cofibrant dg-$A$-module and let $f\colon \mf{g}\rt \mf{h}$ be a weak equivalence between $A$-cofibrant dg-Lie algebroids over $A$. Then restriction along $f$ induces a weak equivalence of graded mixed cdgas $W(\mf{h})\rt W(\mf{g})$.
\end{lemma}
\begin{proof}
For every $A$-cofibrant dg-Lie algebroid $\mf{g}$, the mapping fibre $\tilde{\mf{n}}$ of its anchor map is a cofibrant dg-$A$-module. Furthermore, a weak equivalence between $A$-cofibrant dg-Lie algebroids induces a weak equivalence between their mapping fibres. Observe that the Weil algebra $W(\mf{g})$ is the limit of the sequence of graded mixed cdgas
$$
W(\mf{g})^{\leq p}=\bigoplus_q \Hom_A\Big(\mm{Sym}_A^{\leq p} \mf{g}[1]\otimes \mm{Sym}_A^q\tilde{\mf{n}}, A\Big)
$$
consisting of quotients by polynomials of degree $>p$ in the variables $\mf{g}[1]$. The associated graded of this filtration can be identified with 
$$
\bigoplus_q \ul{\Hom}_A\Big(\mm{Sym}^p_A \mf{g}[1]\otimes \mm{Sym}^q_A\tilde{\mf{n}}, A\Big).
$$
Each of these summands carries a differential obtained from the $A$-linear differentials of $\mf{g}$ and $\tilde{\mf{n}}$ by tensoring and dualising. It follows that $f\colon \mf{g}\rt \mf{h}$ induces a weak equivalence between the associated graded of $W(\mf{h})$ and $W(\mf{g})$, so that the map on limits $W(\mf{h})\rt W(\mf{g})$ is a weak equivalence as well.
\end{proof}
\begin{proposition}
Suppose that $T_A$ is a cofibrant dg-$A$-module and let $\mf{g}$ be a fibrant-cofibrant dg-Lie algebroid over $A$. Restriction along the canonical map $\mf{g}\oplus \ker(\rho)[-1]\rt \mf{g}\oplus \mm{fib}(\rho)[-1]$ induces a weak equivalence of graded mixed cdgas
$$\xymatrix{
W(\mf{g})\ar[r] & C^*(\mc{L}\mf{g}).
}$$
In particular, the Weil algebra of an $A$-cofibrant dg-Lie algebroid $\mf{g}$ computes the (derived) Chevalley-Eilenberg complex of its free loop space $\mc{L}\mf{g}$.
\end{proposition}
\begin{proof}
The second part of the assertion follows from the fact that $W(\mf{g})$ is homotopy invariant, by Lemma \ref{lem:weilishomotopyinvariant}. For the first part, note that $W(\mf{g})$ and $C^*(\mc{L}\mf{g})$ are both filtered by polynomial degree in $\mf{g}[1]$, as in the proof of Lemma \ref{lem:weilishomotopyinvariant}. The restriction map $W(\mf{g})\rt C^*(\mc{L}\mf{g})$ preserves this filtration, and is given on the associated graded by a map
$$\xymatrix@C=1.7pc{
\ul{\Hom}_A\hspace{-1pt}\Big(\hspace{-1pt}\mm{Sym}^p_A \mf{g}[1]\otimes \mm{Sym}^q_A\tilde{\mf{n}}, A\hspace{-1pt}\Big)\ar[r] & \ul{\Hom}_A\hspace{-1pt}\Big(\hspace{-1pt}\mm{Sym}^p_A \mf{g}[1]\otimes \mm{Sym}^q_A\big(\hspace{-1pt}\ker(\rho)\big), A\hspace{-1pt}\Big).
}$$
This map is obtained from the inclusion $\ker(\rho)\rt \mm{fib}(\rho)$ by tensoring and dualising. Since this inclusion is a weak equivalence between cofibrant dg-$A$-modules, the result follows.
\end{proof}
\begin{remark}\label{rem:diffformsvsloops}
Under certain restrictions on $A$, there is an equivalence between the $\infty$-category of dg-Lie algebroids over $A$ and certain `formal derived stacks', or `formal moduli problems', around $\spec(A)$ \cite{nui17c}. Under this equivalence, the Chevalley-Eilenberg complex of $\mf{g}$ corresponds to the dg-algebra of functions on the associated formal derived stack. 
When $C^*(\mf{g})$ is considered as an algebra of functions, the Weil algebra $W(\mf{g})$ closely resembles the corresponding (graded-mixed) complex of \emph{differential forms}. This is not unexpected: in differential geometry, the Weil algebra of an ordinary Lie algebroid is indeed closely related to the complex of differential forms on the associated Lie groupoid \cite{aba11}.

More precisely, recent work in derived geometry \cite{ben12,pan13,toe11} often models differential forms by `functions on (derived) loop spaces', with the de Rham differential coming from loop rotation. For a formal derived stack $X$, this geometric description can be made explicit in terms of algebra, using the computations in this section. Indeed, let $\mf{g}$ denote the dg-Lie algebroid associated to $X$ under the equivalence of \cite{nui17c}. Then the algebra of functions on $\mc{L}X$ is given by $C^*\big(\mc{L}(\mf{g})\big)$, or by the Weil algebra $W(\mf{g})$. By Remark \ref{rem:loops}, the loop rotation can be described concretely by the mixed structure of Lemma \ref{lem:gradedmixedstructure}.
\end{remark}

\bibliographystyle{abbrv}
\bibliography{bibliography_hala}

\providecommand\noopsort[1]{}
\begin{thebibliography}{10}

\bibitem{ale97}
M.~Alexandrov, M.~Kontsevich, A.~Schwarz, and O.~Zaboronsky.
\newblock The geometry of the master equation and topological quantum field
  theory.
\newblock {\em Internat. J. Modern Phys. A}, 12(7):1405--1429, 1997.

\bibitem{aba11}
C.~Arias~Abad and M.~Crainic.
\newblock The {W}eil algebra and the {V}an {E}st isomorphism.
\newblock {\em Ann. Inst. Fourier (Grenoble)}, 61(3):927--970, 2011.

\bibitem{ben12}
D.~Ben-Zvi and D.~Nadler.
\newblock Loop spaces and connections.
\newblock {\em J. Topol.}, 5(2):377--430, 2012.

\bibitem{ber09}
C.~Berger and I.~Moerdijk.
\newblock On the derived category of an algebra over an operad.
\newblock {\em Georgian Math. J.}, 16(1):13--28, 2009.

\bibitem{blu14}
A.~J. Blumberg and E.~Riehl.
\newblock Homotopical resolutions associated to deformable adjunctions.
\newblock {\em Algebr. Geom. Topol.}, 14(5):3021--3048, 2014.

\bibitem{bon13}
G.~Bonavolont\`a and N.~Poncin.
\newblock On the category of {L}ie {$n$}-algebroids.
\newblock {\em J. Geom. Phys.}, 73:70--90, 2013.

\bibitem{bou76}
A.~K. Bousfield and V.~K. A.~M. Gugenheim.
\newblock On {${\rm PL}$} de {R}ham theory and rational homotopy type.
\newblock {\em Mem. Amer. Math. Soc.}, 8(179):ix+94, 1976.

\bibitem{fre09}
B.~Fresse.
\newblock {\em Modules over operads and functors}, volume 1967 of {\em Lecture
  Notes in Mathematics}.
\newblock Springer-Verlag, Berlin, 2009.

\bibitem{gai16}
D.~Gaitsgory and N.~Rozenblyum.
\newblock {\em A study in derived algebraic geometry.}, volume 221 of {\em
  Mathematical Surveys and Monographs}.
\newblock American Mathematical Society, Providence, RI, 2017.

\bibitem{gil07}
J.~Gillespie.
\newblock Kaplansky classes and derived categories.
\newblock {\em Math. Z.}, 257(4):811--843, 2007.

\bibitem{gol88}
W.~M. Goldman and J.~J. Millson.
\newblock The deformation theory of representations of fundamental groups of
  compact {K}\"ahler manifolds.
\newblock {\em Inst. Hautes \'Etudes Sci. Publ. Math.}, (67):43--96, 1988.

\bibitem{hin98}
V.~Hinich.
\newblock D{G} coalgebras as formal stacks.
\newblock {\em J. Pure Appl. Algebra}, 162(2-3):209--250, 2001.

\bibitem{hin93}
V.~Hinich and V.~Schechtman.
\newblock Homotopy {L}ie algebras.
\newblock In {\em I. {M}. {G}el'fand {S}eminar}, volume~16 of {\em Adv. Soviet
  Math.}, pages 1--28. Amer. Math. Soc., Providence, RI, 1993.

\bibitem{hue17}
J.~Huebschmann.
\newblock Multi derivation {M}aurer-{C}artan algebras and sh {L}ie-{R}inehart
  algebras.
\newblock {\em J. Algebra}, 472:437--479, 2017.

\bibitem{kap07}
M.~Kapranov.
\newblock Free {L}ie algebroids and the space of paths.
\newblock {\em Selecta Math. (N.S.)}, 13(2):277--319, 2007.

\bibitem{kje01}
L.~Kjeseth.
\newblock Homotopy {R}inehart cohomology of homotopy {L}ie-{R}inehart pairs.
\newblock {\em Homology Homotopy Appl.}, 3(1):139--163, 2001.

\bibitem{kon03}
M.~Kontsevich.
\newblock Deformation quantization of {P}oisson manifolds.
\newblock {\em Lett. Math. Phys.}, 66(3):157--216, 2003.

\bibitem{laa04}
{\noopsort{Laan}}{P. van der {L}aan}.
\newblock Operads: {H}opf algebras and coloured {K}oszul duality, 2004.
\newblock {T}hesis, available at:
  \href{https://dspace.library.uu.nl/handle/1874/31825}{{U}trecht {U}niversity
  depository}.

\bibitem{lod12}
J.-L. Loday and B.~Vallette.
\newblock {\em Algebraic operads}, volume 346 of {\em Grundlehren der
  Mathematischen Wissenschaften}.
\newblock Springer, Heidelberg, 2012.

\bibitem{lur09}
J.~Lurie.
\newblock {\em Higher topos theory}, volume 170 of {\em Annals of Mathematics
  Studies}.
\newblock Princeton University Press, Princeton, NJ, 2009.

\bibitem{lur11X}
J.~Lurie.
\newblock Derived algebraic geometry {X}: Formal moduli problems, 2011.
\newblock {A}vailable at author's website:
  \url{http://www.math.harvard.edu/~lurie/}.

\bibitem{lur16}
J.~Lurie.
\newblock Higher {A}lgebra, 2016.
\newblock {A}vailable at author's website:\\
  \url{http://www.math.harvard.edu/~lurie/}.

\bibitem{mac87}
K.~Mackenzie.
\newblock {\em Lie groupoids and {L}ie algebroids in differential geometry},
  volume 124 of {\em London Mathematical Society Lecture Note Series}.
\newblock Cambridge University Press, Cambridge, 1987.

\bibitem{nui17c}
J.~Nuiten.
\newblock Koszul duality for {L}ie algebroids.
\newblock {\em
  \href{https://arxiv.org/abs/1712.03442}{\textup{arXiv:1712.03442}}}, 2017.

\bibitem{pan13}
T.~Pantev, B.~To\"en, M.~Vaqui\'e, and G.~Vezzosi.
\newblock Shifted symplectic structures.
\newblock {\em Publ. Math. Inst. Hautes \'Etudes Sci.}, 117:271--328, 2013.

\bibitem{pav14}
D.~Pavlov and J.~Scholbach.
\newblock Admissibility and rectification of colored symmetric operads.
\newblock {\em
  \href{https://arxiv.org/abs/1410.5675}{\textup{arXiv:1410.5675}}}, 2014.

\bibitem{pri10}
J.~P. Pridham.
\newblock Unifying derived deformation theories.
\newblock {\em Adv. Math.}, 224(3):772--826, 2010.

\bibitem{pym16}
B.~Pym and P.~Safronov.
\newblock Shifted symplectic {L}ie algebroids.
\newblock {\em
  \href{https://arxiv.org/abs/1612.09446}{\textup{arXiv:1612.09446}}}, 2016.

\bibitem{rin63}
G.~S. Rinehart.
\newblock Differential forms on general commutative algebras.
\newblock {\em Trans. Amer. Math. Soc.}, 108:195--222, 1963.

\bibitem{roy02}
D.~Roytenberg.
\newblock On the structure of graded symplectic supermanifolds and {C}ourant
  algebroids.
\newblock In {\em Quantization, {P}oisson brackets and beyond ({M}anchester,
  2001)}, volume 315 of {\em Contemp. Math.}, pages 169--185. Amer. Math. Soc.,
  Providence, RI, 2002.

\bibitem{she17}
Y.~Sheng and C.~Zhu.
\newblock Higher extensions of {L}ie algebroids.
\newblock {\em Commun. Contemp. Math.}, 19(3):1650034, 41, 2017.

\bibitem{spi01}
M.~Spitzweck.
\newblock Operads, algebras and modules in general model categories.
\newblock {\em
  \href{https://arxiv.org/abs/math/0101102}{\textup{arXiv:math/0101102}}},
  2001.

\bibitem{toe11}
B.~To\"en and G.~Vezzosi.
\newblock Alg\`ebres simpliciales {$S^1$}-\'equivariantes, th\'eorie de de
  {R}ham et th\'eor\`emes {HKR} multiplicatifs.
\newblock {\em Compos. Math.}, 147(6):1979--2000, 2011.

\bibitem{vez13}
G.~Vezzosi.
\newblock A model structure on relative dg-{L}ie algebroids.
\newblock In {\em Stacks and categories in geometry, topology, and algebra},
  volume 643 of {\em Contemp. Math.}, pages 111--118. Amer. Math. Soc.,
  Providence, RI, 2015.

\bibitem{vit14}
L.~Vitagliano.
\newblock On the strong homotopy {L}ie-{R}inehart algebra of a foliation.
\newblock {\em Commun. Contemp. Math.}, 16(6):1450007, 49, 2014.

\bibitem{sev05}
P.~\v{S}evera.
\newblock Some title containing the words ``homotopy'' and ``symplectic'', e.g.
  this one.
\newblock In {\em Travaux math\'{e}matiques. {F}asc. {XVI}}, volume~16 of {\em
  Trav. Math.}, pages 121--137. Univ. Luxemb., Luxembourg, 2005.

\end{thebibliography}

\end{document}